\documentclass[12pt]{amsart}
\usepackage[run=2]{auto-pst-pdf}
\usepackage{hyperref}
\usepackage{color}

\usepackage{graphicx}
\usepackage{xypic}

\usepackage{amsmath,amsfonts,amssymb}

\usepackage{pstricks}
\usepackage{pst-plot}

\theoremstyle{plain}

\newtheorem{theorem}{Theorem}[section]
\newtheorem{lemma}[theorem]{Lemma}

\newtheorem{proposition}[theorem]{Proposition}

\numberwithin{equation}{section}

\theoremstyle{remark}
\newtheorem{remark}[theorem]{Remark}

\newtheorem{example}[theorem]{Example}
\newtheorem*{claim}{Claim}
\newtheorem*{ack}{Acknowledgement}

\theoremstyle{definition}
\newtheorem{definition}[theorem]{Definition}

\newtheorem{convention}[theorem]{Convention}

\def\Q{\Bbb{Q}}

\def\Z{\Bbb{Z}}
\def\R{\Bbb{R}}

\def\N{\Bbb{N}}
\def\be{\begin{equation}} \def\ee{\end{equation}}

\def\eps{\epsilon}

\def\l{\lambda}
\def\L{\L_{\R}}
\def\ll{\langle}
\def\rr{\rangle}

\def\ext{\operatorname{Ext}}
\def\sm{\setminus}
\def\bp{\begin{pmatrix}}
\def\ep{\end{pmatrix}}
\def\bn{\begin{enumerate}}
\def\en{\end{enumerate}}

\def\order{\operatorname{order}}

\def\ba{\begin{array}}
\def\ea{\end{array}}

\def\L{\Lambda}

\def\s{\sigma}

\def\ti{\widetilde}
\def\lk{\operatorname{lk}}
\def\fr12{\frac{1}{2}}

\def\ol{\overline}

\def\zt{\Z[t^{\pm 1}]}

\def\hom{\operatorname{Hom}}

\def\FH{F\!H}
\def\wti{\widetilde}
\def\lkt{\lk_t}
\def\lktw{\wti{\lk}}

\def\to{\mathchoice{\longrightarrow}{\rightarrow}{\rightarrow}{\rightarrow}}
\makeatletter
\newcommand{\shortxra}[2][]{\ext@arrow 0359\rightarrowfill@{#1}{#2}}
\def\longrightarrowfill@{\arrowfill@\relbar\relbar\longrightarrow}
\newcommand{\longxra}[2][]{\ext@arrow 0359\longrightarrowfill@{#1}{#2}}
\renewcommand{\xrightarrow}[2][]{\mathchoice{\longxra[#1]{#2}}%
  {\shortxra[#1]{#2}}{\shortxra[#1]{#2}}{\shortxra[#1]{#2}}}
\makeatother

\newcommand\IG[2][]{\IfFileExists{#2}{\includegraphics[#1]{#2}}{\fbox{File #2 doesn't exist}\message{Imagefile #2 doesn't exist^^J}}}

\begin{document}

\title{On the algebraic unknotting number}

\thanks{The first author was supported by Polish OPUS grant No 2012/05/B/ST1/03195 and by Fulbright Senior Advanced Research Grant}

\date{\today}

\author{Maciej Borodzik}
\address{Institute of Mathematics, University of Warsaw, Warsaw, Poland}
\email{mcboro@mimuw.edu.pl}

\author{Stefan Friedl}
\address{Mathematisches Institut\\ Universit\"at zu K\"oln\\   Germany}
\email{sfriedl@gmail.com}

\def\subjclassname{\textup{2000} Mathematics Subject Classification}
\expandafter\let\csname subjclassname@1991\endcsname=\subjclassname
\expandafter\let\csname subjclassname@2000\endcsname=\subjclassname


\begin{abstract}
The algebraic unknotting number $u_a(K)$ of a knot $K$ was introduced by Hitoshi Murakami.
It equals the minimal number of crossing changes needed to turn $K$ into an Alexander polynomial one knot.
In a previous paper the authors used the Blanchfield form of a knot $K$ to define an invariant $n(K)$ and
proved that $n(K)\le u_a(K)$.  They also showed that $n(K)$
subsumes all previous classical lower bounds on the (algebraic) unknotting number. In this paper we prove that $n(K)=u_a(K)$.
\end{abstract}

\maketitle

%===========================================
\section{Introduction}

Let $K$ be a knot.
The \emph{unknotting number $u(K)$} is defined to be the minimal number of
crossing changes needed to turn $K$ into the trivial  knot.
The unknotting number is one of the most basic but also most intractable invariant of a knot.
Hitoshi Murakami \cite{Muk90} introduced a more accessible invariant, namely 
the \emph{algebraic unknotting number $u_a(K)$} which is defined to be the minimal number of
crossing changes needed to turn $K$ into a knot with Alexander polynomial equal to one.
(The definition we gave above was shown by Fogel \cite[Theorem~1.4]{Fo93}, see also \cite{Sa99}, to be equivalent to Murakami's original definition
which was given in terms of certain operations on Seifert matrices.)

It is obvious that the algebraic unknotting number is a lower bound on the unknotting number $u(K)$ of a knot.
It is furthermore well--known that the `classical' lower bounds on the unknotting number,
i.e. the lower bounds which can be described in terms of  the Seifert matrix of a knot, like  the Nakanishi index \cite{Na81}, the Levine--Tristram signatures \cite{Mus65,Le69,Tr69,Ta79,BF12}, the Lickorish obstruction \cite{Li85,CL86},
the Murakami obstruction \cite{Muk90} and the Jabuka obstruction \cite{Ja09} give in fact lower bounds on the
algebraic unknotting number.

In \cite{BF11} the authors introduced a new invariant  $n(K)$ of a knot $K$ as follows.
We   write $X(K)=S^3\sm \nu K$ and we consider  the  Blanchfield form
\[ Bl(K)\colon H_1(X(K);\zt)\times H_1(X(K);\zt)\to \Q(t)/\zt.\]
(We refer to Section  \ref{section:seifertbf} for the definition.)
Furthermore, given a
hermitian $n\times n$-matrix $A$ over $\zt$ with $\det(A)\ne 0$  we denote by
$\lambda(A)$ the form
\[ \ba{rcl} \zt^n/A\zt^n \times \zt^n/A\zt^n&\to & \Q(t)/\zt \\
(a,b) &\mapsto & \ol{a}^t A^{-1} b,\ea \]
where we view $a,b$ as represented by column vectors in $\zt^n$.  In \cite{BF11} we defined
\[ n(K):=\mbox{min}
\left\{n\,\left|\,\ba{l}\mbox{there exists a hermitian $n\times n$--matrix $A(t)$ over $\zt$}\\ \mbox{such that $\l(A(t)))\cong Bl(K)$}\\ \mbox{and such that $A(1)$ is diagonalizable over $\Z$}\ea\right. \right\}.\]

In \cite{BF11} we proved that such a matrix $A$ exists, i.e. $n(K)$ is defined, and in fact we showed that $n(K)\leq \deg \Delta_K(t)+1$.
We also proved that $n(K)$ is a lower bound on the algebraic unknotting number, i.e. $n(K)\le u_a(K)$.
We furthermore showed that  $n(K)$ subsumes all the previous classical lower bounds on the unknotting number mentioned above.
In this paper we will now prove that $n(K)$ agrees with the algebraic unknotting number, that is we will show the following theorem:

\begin{theorem}\label{mainthm}
Let $K\subset S^3$ be a knot, then
\[ n(K)=u_a(K).\]
\end{theorem}

In fact in Section \ref{section:proof} we will state and prove a slightly stronger statement which takes into account positive and negative crossing changes.

We have now a following characterization of the algebraic unknotting number.
\begin{proposition}
Let $K\subset S^3$ be a knot. Then the following numbers are equal.
\begin{itemize}
\item[(1)] The algebraic unknotting number, that is the minimal number of crossing needed to turn $K$ into an Alexander polynomial one knot.
\item[(2)] The minimal number of algebraic unknotting moves, see \cite{Muk90,Sa99}, needed to change the Seifert matrix of $K$
into a trivial matrix.
\item[(3)] The minimal second Betti number of a topological 4--manifold that strictly cobounds $M(K)$, the zero framed surgery along $K$, see 
\cite[Definition 2.5]{BF11}.
\item[(4)] The invariant $n(K)$.
\end{itemize}
\end{proposition}
\begin{proof}
Saeki \cite[Theorem~1.1]{Sa99} showed that $(1)=(2)$. In \cite{BF11} it was shown that $(4)\le (3)\le (1)$. By Theorem~\ref{mainthm}
we have actually $(4)=(1)$.
\end{proof}
Note that (2) and (4) are purely algebraic quantities. It would be interesting to find a direct algebraic proof that $(2)=(4)$.

The paper is organized as follows.
In Section \ref{section:bf} we recall the definition of the Alexander module and of the Blanchfield form
using Poincar\'e duality. 
In Section \ref{section:twistedlinkingform}  we then give a more geometric interpretation of the Blanchfield form.

\begin{convention}
  All manifolds are assumed to be oriented and compact, unless it says
  specifically otherwise.
\end{convention}

\begin{ack}
  We would like to thank Micah Fogel for sending us his thesis and
  Hitoshi Murakami for supplying us with a copy of \cite{Muk90}.
  We are also grateful to Jonathan Hillman for providing us with a proof of Lemma \ref{lem:free}.
  Furthermore we would like to express our gratitude to the Renyi
  Institute for its hospitality. The first author wishes also to Indiana University and Universit\"at zu K\"oln for hospitality and Fulbright Foundation
for supporting his visit in Indiana University. The second author wishes to thank Warsaw University for hospitality.
\end{ack}

%===========================================
\section{The Blanchfield form}\label{section:bf}

%===========================================
\subsection{Homologies of complexes over $\zt$}

Let $C_*$ be any chain complex of finitely generated free $\zt$--modules and let $M$ be any $\zt$-module.
We can then consider the corresponding homology and cohomology modules:
\begin{equation}\label{eq:homology}
\ba{rcl} H_*(C;M)&:=&H_*(C_*\otimes_{\zt}M),\mbox{ and }\\
H^*(C;M)&:=&H_*(\hom_{\zt}(C_*,M)).\ea
\end{equation}
By \cite[Theorem~2.3]{Lev77} there is a spectral sequence 
$E_{p,q}^r$
with
\[E_{p,q}^2=\ext^p_{\zt}(H_q(C),M)\] 
and  which converges to  $H^*(C,M)$.
This spectral sequence is  called the Universal Coefficient Spectral Sequence, or UCSS for short.
We note that for any two $\zt$-modules $H$ and $M$
the module $\ext^0_{\zt}(H,M)$ is canonically isomorphic to
$\hom_{\zt}(H,M)$.

Also note that
\[ \ext^p_{\zt}(H,M)=0\]
for any $p>2$ since $\zt$ has cohomological dimension 2.
Finally note, that if $\Z$ is considered as a $\zt$ module
with trivial $t$--action, then $\Z$ admits a resolution of length $1$, in particular
\[ \ext^p_{\zt}(\Z,M)=0\]
for any $p>1$.

For later use we also record the following lemma.

\begin{lemma}\label{lem:ka86} \label{lem:free}
Let $H$ be a finitely generated $\zt$--module, then $\ext^0_{\zt}(H,\zt)$
is a free $\zt$--module.
\end{lemma}

The lemma is well-known but where are not aware of a good reference. We thus provide a short proof, whose key idea was supplied to us by Jonathan Hillman.

\begin{proof}
Let $H$ be a finitely generated $\zt$--module. Since $\zt$ is Noetherian there exists an exact sequence of the form
$\zt^r\xrightarrow{\varphi} \zt^s\to H$.
Since the Hom-functor $M\mapsto \hom_{\zt}(M,\zt)$ is left-exact the above exact sequence gives rise to  an exact sequence 
\[ \ba{rclcl} 0&\to &\hom(H,\zt)&\to& \hom(\zt^s,\zt)\xrightarrow{\varphi^*}\hom(\zt^r,\zt)\\
&\to& \operatorname{coker}(\varphi^*)&\to& 0.\ea \]
Note that  $\zt$ is a ring of homological dimension 2.
(This is for example a straightforward consequence of the fact that the ring $\Z[t]$ has homological dimension 2 which is proved in \cite[Theorem~5.36]{La06}.) We can therefore find a projective resolution
\[ 0\to P_2\to P_1\to P_0\to \operatorname{coker}(\varphi^*)\to 0\]
for $\operatorname{coker}(\varphi^*)$ of length two. 
Comparing these two resolutions for $\operatorname{coker}(\varphi^*)$ and noting that  $\hom(\zt^s,\zt)$ and $\hom(\zt^r,\zt)$ are free $\zt$-modules implies by
Schanuel's lemma (see \cite[Corollary~5.5]{La99}) that $\hom(H,\zt)$ is projective.
Finally, it is a special case of the Serre Conjecture, see e.g. \cite[Corollary~4.12]{La06}, that a finitely generated projective $\zt$-module is in fact free. This concludes the proof that $\hom(H,\zt)$ is a free $\zt$-module.
\end{proof}

%===========================================
\subsection{Twisted homology, cohomology groups and Poincar\'e duality}

Let $X$ be a topological space and let $\phi\colon \pi_1(X)\to \ll t\rr$ be an epimorphism onto the infinite cyclic group generated by $t$.
We denote by $\pi\colon \wti{X}\to X$  the corresponding infinite cyclic  covering of $X$.
Given a subspace $Y\subset X$ we write $\wti{Y}:=\pi^{-1}(Y)$.

The deck transformation induces a canonical $\zt$--action on $C_*(\wti{X},\wti{Y};\Z)$
and we can thus view $C_*(\wti{X},\wti{Y};\Z)$ as a chain complex of free $\zt$--modules.
Now let $M$ be a module over $\zt$. We then consider homologies $H_*(X,Y;M)$ and $H^*(X,Y;M)$
as defined in \eqref{eq:homology}. The most important instance will be $M=\zt$.

If $K\subset S^3$ is an oriented knot, then we denote by $\phi\colon \pi_1(X(K))\to \ll t\rr$ the epimorphism given by sending the oriented meridian to $t$.
Furthermore, if $X$ is a space with $H_1(X;\Z)\cong \Z$, then we pick either epimorphism from $\pi_1(X)$ onto $ \ll t\rr$.
For different choices of epimorphisms the resulting  modules $H_*(X,Y;\zt)$ and $H^*(X,Y;\zt)$
will be anti--isomorphic, i.e. multiplication by $t$ in one module corresponds to multiplication by $t^{-1}$ in the other module. 
Since this does not affect any of the arguments we will usually not record the choice of $\phi$ in our notation.

Finally suppose that  $X$ is an orientable $n$-manifold and that $W$ is union of components of $\partial X$.
Then for any $\zt$-module $M$,
Poincar\'e duality (see e.g. \cite[Chapter~2]{Wa99}) defines isomorphisms of $\zt$-modules
\[  H_i(X,W;M)\cong \ol{H^{n-i}(X,\partial X\sm W;M)},\]
in particular if $W=\emptyset$, then we get a canonical isomorphism
\[ H_i(X;M)\cong \ol{H^{n-i}(X,\partial X;M)}. \]
Here, given a $\zt$--module $N$ we denote by $\ol{N}$ the same abelian group as $N$ but with the involuted $\zt$--action, i.e. multiplication by $t$ on $\ol{N}$ corresponds to multiplication by $t^{-1}$ on $N$.

%===========================================
\subsection{Orders of $\zt$--modules}

Let $H$ be a finitely generated $\zt$--module. Since $\zt$ is Noetherian it follows that $H$ is also finitely presented, i.e.  we can find a resolution
\[ \zt^m\xrightarrow{A}\zt^n\to H,\]
where we can assume that $m\geq n$. We then define $\order(H)\in \zt$ to be the greatest common divisor of the $n\times n$--minors of $A$. It is well--known that, up to multiplication by a unit in $\zt$, i.e. up to multiplication by an element of the form $\pm t^k, k\in \Z$, the invariant $\order(H)$ is independent of the choice of $A$.
 We refer to \cite{Hi02} for details.
In the following, given $f,g\in \zt$ we write
$f\doteq g$ if $f$ and $g$ agree up to multiplication by a unit in $\zt$.

\begin{example}\label{presentation}
  If $H$ admits a square presentation matrix $A$ over $\zt$ of size
  $n$, then it follows immediately from the definition that the order
  of $H$ equals $\det(A)$.
\end{example}

\begin{example}
  The Alexander polynomial of a knot $K$ is defined to be the order of
  the Alexander module $H_1(X(K);\zt)$. Throughout this paper we will
  normalize the Alexander polynomial such that $\Delta_K(1)=1$ and
  $\Delta_K(t^{-1})=\Delta_K(t)$.
\end{example}

The following result is standard (see e.g. \cite[Section~3]{Hi02}), we will use it often in the future.

\begin{lemma}\label{annihilates}
 The order of any $\zt$--module is also an annihilator, i.e. $\order(H)\cdot v=0$ for any $v\in H$. In
 particular if $K$ is knot, then  for any $c\in H_1(X(K);\zt)$, we have $\Delta_K(t)\cdot c=0$.
\end{lemma}

 \begin{remark}
Given $p=p(t)\in \zt$ we define $\ol{p}:=p(t^{-1})$. Note that for any $\zt$--module we have
\[ \order(\ol{M})\doteq \ol{\order(M)}.\]
\end{remark}

We will later make use of the following lemma (see again \cite{Hi02} for details).

\begin{lemma}\label{lem:orders}
Let
\[ 0\to H\to H'\to H''\to 0\]
be a short exact sequence of $\zt$--modules, then
\[ \order(H')\doteq \order(H)\cdot \order(H'').\]
\end{lemma}

%===========================================
\subsection{The homological definition of the Blanchfield form}\label{section:seifertbf}

Let $K\subset S^3$ be a knot.
We  consider the following sequence of maps:
\[
\ba{rcl}
\Phi\colon H_1(X(K);\zt)&\to& H_1(X(K),\partial X(K);\zt)\\[2mm]
&\to& \ol{H^2(X(K);\zt)}\xleftarrow{\cong} \ol{H^1(X(K);\Q(t)/\zt)}\\[2mm]
&\to &\ol{\hom_{\zt}(H_1(X(K);\zt),\Q(t)/\zt)}.\ea \]
Here the first map is the inclusion induced map, the second map is Poincar\'e duality,
the third map comes from the long exact sequence in cohomology corresponding to the coefficients
$0\to \zt\to \Q(t)\to \Q(t)/\zt\to 0$, and the last map is the evaluation map.
All these maps are isomorphisms, and hence define a non-singular form
\[ \ba{rcl} Bl(K)\colon  H_1(X(K);\zt)\times H_1(X(K);\zt)&\to& \Q(t)/\zt\\
(a,b)&\mapsto & \Phi(a)(b),\ea \]
called the \emph{Blanchfield form of $K$}. This form is well-known to be hermitian, in particular   $Bl(K)(a_1,a_2)=\ol{Bl(K)(a_2,a_1)}$ and $Bl(K)(\mu_1 a_1,\mu_2a_2)=\ol{\mu_1}Bl(K)(a_1,a_2)\mu_2$
for $\mu_i\in \zt, a_i\in H_1(X(K);\zt)$.
The Blanchfield form was initially introduced by Blanchfield \cite{Bl57}. We will give a more geometric definition in the next section.

\begin{remark}
By Lemma~\ref{annihilates} the polynomial  $\Delta_K(t)$ annihilates $H_1(X(K);\zt)$, it follows easily from the definitions that 
$Bl(K)$ takes in fact values in $\Delta_K(t)^{-1}\zt/\zt\subset \Q(t)/\zt$.
\end{remark}

%===========================================
\section{The twisted linking form}\label{section:twistedlinkingform} 

\subsection{Pairings on infinite cyclic covers}\label{section:blanchfield}

Let $K\subset S^3$ be an oriented knot.
We write $X=X(K)$, which we endow with the orientation coming from $S^3$, and we denote by $\Delta$ the Alexander polynomial of $K$.
Recall that $\phi\colon \pi_1(X)\to \ll t\rr$ is the unique epimorphism which sends the oriented meridian of $K$ to $t$.
Then $\ll t\rr$ acts on $\wti{X}$, the corresponding infinite cyclic cover of $X$;
we can thus view $H_i(\wti{X})$ as a $\zt$--module. This module is by definition precisely the 
Alexander module $H_i(X;\zt)$ as defined above.

We say that a simple closed curve $c\subset \wti{X}$ is in \emph{general position}
if $t^ic$ and $c$ are disjoint for any $i\in \Z$.
Furthermore we say that a pair of simple closed oriented curves $c,d$ is in general position in $\wti{X}$,
if
 $t^ic$ and $d$ are disjoint for any $i\in \Z$.
 Finally, if $c$ is a simple closed curve and $F$ an embedded surface in $\wti{X}$, then we say that they are in general position if for any $i\in \Z$ the curve $t^ic$ intersects $F$ transversely.

If $c$ is a simple closed oriented curve in $\wti{X}$ and $n\in \N$, then we denote by $nc$ the union of $n$ parallel copies of $c$.
We can and will assume that these parallel copies are in general position to each other.
If $-n\in \N$, then we denote by $nc$ the union of $-n$ parallel copies of $-c$, i.e. of $c$ with opposite orientation.
Finally if $p(t)=\sum_{i=k}^la_it^i\in \zt$, then we denote by $p(t)c$ the union of $a_kt^kc\cup \dots \cup a_lt^lc$.

The following definition is now a variation on the equivariant intersection number in a covering space (see e.g. \cite[p.~495]{COT03}).

\begin{definition}
Let $c,d\subset\wti{X}$ be simple closed oriented curves in general position.
By Lemma~\ref{annihilates} 
there exists an embedded oriented  surface $F\subset \wti{X}$  such that $\partial F=\Delta\cdot c$.
 We can arrange that $F$ and $d$ are in general position.
The \emph{twisted linking number}
of $c$ and $d$ is defined as
\begin{equation}\label{eq:twisted}
\lktw(c,d):= \frac{1}{\Delta}\sum_{i\in \Z}(F\cdot t^id)\cdot t^{-i}\in \frac{1}{\Delta}\zt.
\end{equation}
Here $F\cdot t^id$ denotes the ordinary intersection number of the oriented submanifolds $F$ and $t^id$ in $\wti{X}$.
\end{definition}

\begin{lemma}
The twisted linking form $\lktw(c,d)$ is independent of the choice of $F$.
\end{lemma}

\begin{proof}
  By Poincar\'e duality we have
  \[ H_2(X;\zt)\cong \ol{H^1(X,\partial X;\zt)},\] but $H_1(X,\partial
  X;\zt)$ is $\zt$--torsion and $H_0(X,\partial X;\zt)=0$. It now
  follows from the UCSS that $H_2(\wti{X};\Z)=H_2(X;\zt)=0$. Now let
  $F'$ be any other surface cobounding $\Delta\cdot c$, then $F\cup
  -F'$ forms a closed oriented surface in $\wti{X}$, in particular it
  represents an element in $H_2(X;\zt)$.  But since $H_2(X;\zt)=0$ it
  now follows that $(F\cup -F')\cdot d=0$.  This concludes the proof
  of the lemma.
\end{proof}

\begin{lemma}
\[ \lktw(d,c)=\ol{\lktw(c,d)}.\]
\end{lemma}

\begin{proof}
  Let $F,G\subset \wti{X}$ be embedded oriented surfaces such that
  $\partial F=\Delta\cdot c$ and $\partial G=\Delta \cdot d$. We can
  assume that $t^iF$ intersects $G$ transversely for any $i$.  For any
  $i$ the 1--manifold $t^iF\cap G$ defines a cobordism between
  $t^iF\cap d$ and $G\cap t^ic$.  It thus follows that
  \[ \ba{rcl} \Delta \cdot \lktw(d,c)&=&\sum_{i\in \Z} (G\cdot t^ic)t^{-i}=\sum_{i\in \Z} (t^iF\cdot d)t^{-i}=\\[2mm]
  &=&\sum_{i\in \Z} (F\cdot t^{-i}d)t^{-i}=\sum_{i\in \Z} (F\cdot
  t^{i}d)t^{i}=
  \ol{\sum_{i\in \Z} (F\cdot t^{i}d)t^{-i}}=\\[2mm]
  &=& \ol{ \Delta \cdot \lktw(c,d)}= \ol{\Delta} \cdot \ol{\lktw(c,d)}= \Delta
  \cdot \ol{\lktw(c,d)}.\ea \]
\end{proof}

In general, if $c$ and $c'$ are homologous curves in $\wti{X}$, the linking form $\lktw(c,d)$ and $\lktw(c',d)$
will be different (unless $c$ and $c'$ are homologous in $\wti{X}\setminus d$). Nevertheless, $\lktw(c,d)\bmod \zt$
is homology invariant. Therefore, $\lktw(c,d)$ 
descends to a form
\[ H_1(X;\zt)\times H_1(X;\zt)\to \frac{1}{\Delta}\zt/\zt,\]
which by definition is precisely the Blanchfield form $Bl(K)$.
We refer to \cite{Bl57} for details.

%===========================================
\subsection{Based curves and surfaces}\label{section:basedcs}

In this section we will take a point of view which differs from the  discussion in the previous section:
instead of studying objects in the infinite cyclic cover of $X(K)$ we will now consider based objects in $X(K)$.

Let $K\subset S^3$ be an oriented knot. As above we write $X=X(K)$ and we denote the infinite cyclic cover of $X$ by $\wti{X}$.
In this section we will define an invariant $\lkt$ which will turn out to capture the same information as $\lktw$ in the previous section, 
but instead of considering curves in $\wti{X}$
we will now work with based curves in $X$.

We fix once and for all a base point $*$ in $X$. We now need several definitions:
\bn
\item By a \emph{surface} in $X$ we always mean an immersed surface. By a \emph{smooth curve} on the immersed surface we mean the image of a smooth curve on the original surface under the immersion.
\item A \emph{based} curve (respectively surface) in $X$ is an oriented curve (respectively oriented surface) in $X$ together with a path, called
\emph{basing}
connecting it to the base point $*$. We assume that the basing intersects the curve (respectively the surface) in only one point.
\item By an orientation of a based curve (respectively surface) we mean an orientation of the unbased curve (respectively surface).
\item A curve $c$ in $X$ is called \emph{homologically trivial} if $c$ is trivial in $H_1(X;\Z)$.
\item A surface $F$ in $X$ is \emph{homologically invisible}
if any smooth curve on $F$ is null--homologous in $X$.  Note that a curve (respectively surface) is homologically trivial if and only if
it lifts to $\wti{X}$.
\item We say that two based homologically trivial curves are \emph{equivalent} if the unbased curves agree
 and if the basings are homologous relative to the base point
and relative to a path connecting the end points on the curve. (This condition does not depend on the path since the curve is assumed to be homologically trivial.)
Similarly we define equivalence of based homologically invisible surfaces.

\item We say that two based objects are disjoint if the corresponding unbased objects are disjoint.
\item We say that a based curve $c$ and a based surface  $F$ in  $X$ are in \emph{general position}  if the unbased curve and the unbased surface
are in general position and if furthermore the basings are embedded and disjoint from $c$ and from $F$.
\en

Let $c$ be a homologically trivial based curve in $X$ and let $F$ be a homologically invisible based surface in $X$ such that $F$ and $c$ are in general position.
Any intersection point $P$ of the (unbased) curve and the (unbased) surface  comes with a sign $\eps_P\in \{-1,1\}$. To any intersection point $P$ we
can also associate a loop $l_P$ in $X$ in the following way.
We go from the base point $*$ via a smooth curve on  the based surface $F$ to the intersection $P$, and then we go back to $*$ along the curve $c$.
Since $F$ is homologically invisible and $c$ is homologically trivial, it follows that $\phi(l_P)$ is independent of the choices.
Following \cite[p.~499]{COT03} 
we now define
\[ F\cdot c:=\sum_{P\in c\cap F} \eps_P\,\phi(l_P)\in \zt.\]
Note that $F\cdot c$ only depends on the equivalence classes of $F$ and $c$.
We will thus in the following mostly consider based curves and surfaces up to equivalence.

Given a based curve $c$ and $k\in \Z$ we now denote by $t^kc$ the based curve which is given by precomposing the basing with a closed loop $l$
which satisfies $\phi(l)=t^k$. Note that the equivalence class of $t^kc$ is well--defined.
Furthermore, given $n\in \Z$ we denote by $nc$ the union of $|n|$ parallel copies of $c$, with opposite orientation if $n<0$.
For any Laurent polynomial $p(t)\in \zt$ we define $p(t)c$ in the obvious way. Obviously
\[F\cdot p(t)c=p(t)(F\cdot c).\]

Let $F$ be a based homologically invisible surface. Its boundary components inherit basings which are well--defined up to equivalence.
We can thus view $\partial F$ as a union of based curves.

We denote the infinite cyclic covering map of $X$ by $\pi\colon \wti{X}\to X$ and we pick a base point $\wti{*}$ in $\wti{X}$ lying over $*$.
With these choices there is a one-to-one correspondence
\[ \ba{c}\mbox{equivalence classes of}\\
\mbox{based curves (surfaces) in $X$}\ea \Leftrightarrow \ba{c}\mbox{curves (surfaces) in $\wti{X}$.}\ea \]
Now let $c,d$ be based curves which only intersect at $*$.
Then the corresponding  closed curves $\ti{c},\ti{d}$ in $\wti{X}$ are in general position.

By Lemma~\ref{annihilates}, there exists a surface $\wti{F}\subset \wti{X}$ such that $\partial \wti{F}=\Delta\ti{c}$. Let us choose
a curve $\wti{\gamma}$ connecting $\wti{*}$ to a point on $\wti{F}$. The projection of $\wti{F}$ to $X$ yields an immersed surface $F\subset X$. Then $F$
is a based surface, the basing is $\gamma$, a projection of $\wti{\gamma}$ to $X$.

Any smooth curve on $F$ is an image of a curve on $\wti{F}$ by definition. In particular, any smooth curve on $F$ lifts to $\wti{X}$, which means that
$F$ is homologically invisible.
By construction $\partial F=\Delta c$.
We can now define
\[ \lkt(c,d):=\frac{1}{\Delta}F\cdot d\in \frac{1}{\Delta}\zt.\]
It is straightforward to see that
\[ \lkt(c,d)=\lktw(\ti{c},\ti{d})\in \frac{1}{\Delta}\zt.\]
It thus follows from the previous section that $\lkt(c,d)$ is well--defined and that it satisfies
$\lkt(d,c)=\ol{\lkt(c,d)}$. It also follows easily from the definitions that
\[  \lkt(c,d)|_{t=1}\,\,=\,\mbox{lk}(c,d),\]
i.e. the evaluation of $\lkt(c,d)$ at $t=1$ equals the linking number of the unbased curves $c$ and $d$.
Finally note, that  $\lkt(c,d)$ is an invariant of the isotopy class of $c\cup d$.
This follows from the definitions and the fact that any isotopy of $c\cup d$ extends to an isotopy of $S^3$.

From now on we shall use only the notation $\lkt(c,d)$.

By a \emph{framed curve} in $X$ we mean a pair $(c,m)$ where $c$ is a based simple closed curve and $m\in \Z$.
Given such $(c,m)$ we now define
\[ \lkt((c,m),(c,m)):=\lkt(c,c'),\]
where $c'$ is a longitude of $c$ with the property that $\lk(c,c')=m$.
It follows immediately from the above that
\[ \lkt((c,m),(c,m))|_{t=1} \,\,=\,\, \lkt(c,c')|_{t=1}\,\,=\,\,\mbox{lk}(c,c')=m.\]
If $n\ne m$, then 
Also note that if we equip $c$ with framing $n$ instead, then
\[\lkt((c,n),(c,m))=\lkt((c,m),(c,m))+n-m.\]
In the following  we will often suppress $m$ and we will just say that $c$ is a based simple closed curve with framing $m$. In particular if the framing is understood, then we will just write $\lkt(c,c)$.
Also, if $c=(c,m)$ and $d=(d,n)$ are framed curves, such that $c$ and $d$ are disjoint, then we define
\[ \lkt((c,m),(d,n)):=\lkt(c,d).\]

%=============================
\section{4--manifolds and intersection forms}

%===========================================
\subsection{The twisted intersection form}

In the following let $W$ be a 4--manifold, possibly with boundary, with the following properties:
\bn
\item $H_1(W;\Z)\cong \Z$,
\item $H_1(W;\zt)=0$,
\item $\FH_2(W;\zt):=H_2(W;\zt)/\{\mbox{$\zt$--torsion}\}$ is a free $\zt$--module.
\en
We now define the intersection form $Q_W$ on $\FH_2(W;\zt)$. First consider the sequence of maps
\be \label{equ:defqw} \ba{rcl} \Psi\colon H_2(W;\zt)&\to& H_2(W,\partial W;\zt)\xrightarrow{\cong} \ol{H^2(W;\zt)}\\
&\to& \ol{\hom_{\zt}(H_2(W;\zt),\zt)},\ea \ee
where the first map is the inclusion induced map, the second map is Poincar\'e duality and the third map is the evaluation map.
The second map is evidently an isomorphism. The third map is also an isomorphism,
indeed, since $H_1(W;\zt)=0$ and since $\ext_{\zt}^i(\Z,\zt)=0$ for $i>1$ we see that the UCSS for $H^2(W;\zt)$ collapses, i.e. the evaluation map
\[ H^2(W;\zt)\to \hom_{\zt}(H_2(W;\zt),\zt)\] is in fact an isomorphism.
In constrast, the first map in (\ref{equ:defqw}) is in general not an isomorphism.

From (\ref{equ:defqw}) we now obtain a form
\[ \ba{rcl} H_2(W;\zt)\times H_2(W;\zt)&\to& \zt\\
(a,b)&\mapsto & \Psi(a)(b)\ea \]
but this clearly descends to a form
\[ \FH_2(W;\zt)\times \FH_2(W;\zt)\to \zt,\]
which we denote by $Q_W$.  The form $Q_W$ can also be defined more geometrically using equivariant intersection numbers of immersed based surfaces. This interpretation then quickly shows that $Q_W$ is hermitian. We refer to \cite[Chapter~5]{Wa99} for details.

We now pick a basis for the free $\zt$--module $\FH_2(W;\zt)$ and we denote by $\det(Q_W)$ the matrix of the intersection form $Q_W$ with respect to this basis.
Note that the determinant is in fact well--defined, that is, up to a unit in $\zt$ it does not depend on the choice of basis for $\FH_2(W;\zt)$.
The following lemma shows, that one can also determine $\det(Q_W)$ using any maximal set of linearly independent vectors in $\FH_2(W;\zt)$, not
necessarily a basis.

\begin{lemma}\label{lem:detqw}
 Let $v_1,\dots,v_n\in \FH_2(W;\zt)$ be a maximal set of linearly independent vectors in $\FH_2(W;\zt)$.
We denote by $f\in \zt$ the order of the $\zt$--module
\[ \FH_2(W;\zt)/(v_1,\dots,v_n),\]
then
\[ \det(Q_W)\cdot f\cdot \ol{f}\doteq  \det\left(\{Q_W(v_i,v_j)\}_{ij}\right).\]
\end{lemma}

\begin{proof}
%Note that the intersection form $Q_W$ extends to a hermitian form
%\[ H_2(W;\Q(t))\times H_2(W;\Q(t))\to \Q(t).\]
%Now let  $u_1,\dots,u_n$ be any basis of $H_2(W;\Q(t))$ and let $w_1,\dots,w_n$ be a basis
%of $\FH_2(W;\zt)$. Note that $w_1,\dots,w_n$ also defines a basis for $H_2(W;\Q(t))$ and we have
Since $\FH_2(W;\zt)$ is free, there is a basis $w_1,\dots,w_n$. The vectors $v_1,\ldots,v_n$ can be expressed in terms of $w_1,\ldots,w_n$.
Let $P$ be an $n\times n$ matrix over $\zt$, such that $Pv_j=w_j$ for any $j=1,\ldots,n$. We have
\[ \det(Q_W(v_i,v_j)_{ij})=\ol{\det(P)}\det(Q_W(w_i\cdot w_j)_{ij})\det(P)\doteq\det(Q_W)\cdot \det(P)\cdot \ol{\det(P)}.\]
We claim that $f\doteq\det(P)$. Indeed, $P$ can be regarded as a map $\zt^n\to\zt^n$. On the one hand, $\det P$ is the order of the cokernel 
(see Example~\ref{presentation}).
On the other hand, the cokernel of $P$ is $\FH_2(W;\zt)/(v_1,\dots,v_n)$.
\end{proof}

%===========================================
\subsection{$\zt$--cobordisms}

We say that a 3--manifold $M$ is a \emph{homology $S^1\times S^2$}
if $M$ is closed, if $H_1(M;\Z)=\Z$
and if $M$ comes equipped with a choice of an isomorphism $H_1(M;\Z)\to \Z$.
Given a 3--manifold $M$ which is a homology $S^1\times S^2$ we can consider the module $H_1(M;\zt)$,
and we can define a Blanchfield form on $H_1(M;\zt)$
in the same fashion as for $X(K)$.
We denote by $\Delta_M=\Delta_M(t)$ the order of $H_1(M;\zt)$. Note that 
$H_1(M;\Z)=\Z$ implies that $\Delta_M(1)=1$, in particular $\Delta_M(t)$ is non-zero.
The standard arguments already employed for $X(K)$ show that
\[ H_2(M;\zt)\cong \ol{H^1(M;\zt)}\cong \ol{\ext^0_{\zt}(\Z,\zt)}\cong \Z\]
is in fact isomorphic to the trivial $\zt$--module $\Z$.

\begin{example}
Let  $K$ be a knot. We denote by $M(K)$ the zero--framed surgery on $K$.
The inclusion map $X(K)\to M(K)$ induces an isomorphism $H_1(X(K);\Z)\to H_1(M(K);\Z)$.
Together with the isomorphism $H_1(X(K);\Z)\to \Z$ sending an oriented meridian to one
we get a preferred isomorphism $H_1(M(K);\Z)\to \Z$.
It follows that  $M(K)$ is a homology $S^1\times S^2$. It is well--known that the inclusion $X(K)\to M(K)$
induces an isomorphism $H_1(X(K);\zt)\to H_1(M(K);\zt)$, which is in fact an isometry of the Blanchfield forms.
\end{example}

\begin{definition}
Let $M$ and $M'$ be 3--manifolds which are homology $S^1\times S^2$'s.
By a \emph{$\zt$--cobordism between $M$ and $M'$} we understand an orientable, compact 4--manifold $W$ with the following properties:
\bn
\item $\partial W=M\cup -M'$,
\item $H_1(M;\Z)\to H_1(W;\Z)$ and $H_1(M';\Z)\to H_1(W;\Z)$ are isomorphisms, and the following diagram given by the inclusions
and the preferred isomorphisms commutes:
\[ \xymatrix{ H_1(M;\Z)\ar[dr]\ar[r] &H_1(W;\Z)&\ar[dl]\ar[l] H_1(M';\Z) \\ &\Z,&}\]
\item $H_1(W;\zt)=0$.
\en
\end{definition}

We now have the following lemma:

 \begin{lemma}
Let $M$ and $M'$ be 3--manifolds which are homology $S^1\times S^2$'s.
Let $W$ be a $\zt$--cobordism between $M$ and $M'$, then
the following $\zt$--modules are free:
\bn
\item $H_2(W,M;\zt)$ and $ H_2(W,M';\zt)$,
\item $ H_2(W,\partial W;\zt)$, and
\item $\FH_2(W;\zt)=H_2(W;\zt)/\mbox{$\zt$--torsion}.$
\en
\end{lemma}

\begin{proof}
(1)
We first consider $H_2(W,M;\zt)$. By Poincar\'e duality this is isomorphic to $\ol{H^2(W,M';\zt)}$.
The long exact sequence in $\zt$--homology of the pair $(W,M')$ yields:
\[ \ba{ccccccccc} &&&H_1(W;\zt)&\to& H_1(W,M';\zt)&\to&\\
& H_0(M';\zt)&\to & H_0(W;\zt)&\to& H_0(W,M';\zt)&\to& 0.\ea \]
Our assumptions on $W$ imply that $H_1(W;\zt)=0$ and that $H_0(M';\zt)\to H_0(W;\zt)$ is an isomorphism.
We thus conclude that
\[ H_1(W,M';\zt)=H_0(W,M';\zt)=0.\]
The UCSS implies that
\[ H^2(W,M';\zt)\cong \hom_{\zt}(H_2(W,M';\zt),\zt),\]
but from  Lemma \ref{lem:ka86} it follows that $\hom_{\zt}(H_2(W,M';\zt),\zt)$ is a free $\zt$--module.
We infer that  $H_2(W,M;\zt)$ is a free $\zt$--module. The same argument shows of course that
$H_2(W,M';\zt)$ is also free.

(2) By Poincar\'e duality we have an isomorphism
\[ H_2(W,\partial W;\zt)\cong \ol{H^2(W;\zt)}.\]
Since  $H_1(W;\zt)=0$ by assumption and since $\ext_{\zt}^i(\Z,\zt)=0$ for $i>1$ it
follows from the  UCSS, that $H^2(W;\zt)\cong \hom_{\zt}(H_2(W;\zt),\zt)$,
which is free by   Lemma \ref{lem:ka86}.

(3) Finally we  want to show that $\FH_2(W;\zt)$ is also free.
Recall that  by assumption $H_1(W;\zt)=0$.
We obtain the following exact sequence
\[ \ba{cccccccc} &H_2(M;\zt)&\to& H_2(W;\zt)&\to &H_2(W,M;\zt)&\to\\
\to & H_1(M;\zt)&\to& 0.\ea \]
Note that $H_2(M;\zt)$ is $\zt$--torsion and $H_2(W,M;\zt)$ is a free $\zt$--module by the above, in particular the module $H_2(W,M;\zt)$ is $\zt$--torsion free.
The above exact sequence thus descends to the
 following short exact sequence
\begin{equation}\label{exactsequence}
0\to \FH_2(W;\zt)\to H_2(W,M;\zt)\to H_1(M;\zt)\to 0.
\end{equation}
Since $H_2(W,M;\zt)$ is free we can find an isomorphism
\[ \Phi\colon \zt^n\to H_2(W,M;\zt)\]  for some appropriate $n$.

Now let $v_1,\dots,v_m$ be a minimal generating set for $\FH_2(W;\zt)$. We thus obtain the following commutative diagram of exact sequences:
\[ \xymatrix{ &\zt^m\ar[d]^\Psi\ar[r]^A & \zt^n \ar[d]^\cong_\Phi \ar[r] & H_1(M;\zt)\ar[r]\ar[d]^= &0 \\
0\ar[r] &\FH_2(W;\zt)\ar[r]^d & H_2(W,M;\zt)\ar[r] &H_1(M;\zt)\ar[r] &0,}\]
where $\Psi$ sends the $i$-th standard basis vector of $\zt^m$ to $v_i$
and where $A$ is given by $\Phi^{-1}\circ d\circ \Psi$.
The $n\times m$--matrix $A$ over $\zt$ is thus a presentation matrix for $H_1(M;\zt)$.
It is well--known that $H_1(M;\zt)$ admits a square presentation matrix $B$, e.g. we can take $B=Vt-V^t$, where $V$ denotes a Seifert matrix.
Note that $\det(B)=\Delta_K(t)$ is non--zero, i.e. the columns of $B$ are linearly independent over $\zt$.

It now follows from \cite[Theorem~6.1]{Li97}, that
\[ \ba{ll} &\mbox{minimal number of generators of column space of $A$}\,\,-\,\,\mbox{number of rows of $A$}\\
=& \mbox{minimal number of generators of column space of $B$}\,\,-\,\,\mbox{number of rows of $B$}.\ea \]
The latter is zero by the above, so we see that $m=n$.
Since $A$ is therefore a square matrix we see that $\det(A)=\Delta_K(t)$,
in particular the map given by the matrix $A$ is injective.

 We thus obtain the following commutative diagram of short exact sequences:
\[ \xymatrix{ 0\ar[r] &\zt^n\ar[d]\ar[r] & \zt^n \ar[d]^\cong_\Phi \ar[r] & H_1(M;\zt)\ar[r]\ar[d]^= &0 \\
0\ar[r] &\FH_2(W;\zt)\ar[r] & H_2(W,M;\zt)\ar[r] &H_1(M;\zt)\ar[r] &0.}\]
It now follows from the 5--lemma that the vertical map $\zt^n\to \FH_2(W;\zt)$ is an isomorphism, in particular $\FH_2(W;\zt)$ is free.
\end{proof}

The following result is one of the two homological ingredients in the proof of Theorem~\ref{mainthm}.

\begin{proposition}\label{prop:deltakj}
Let $K$ and $J$ be knots in $S^3$ and let $W$ be a $\zt$--cobordism between $M(K)$ and $M(J)$,
then
\[ \det(Q_W)\doteq\Delta_K(t)\cdot \Delta_J(t).\]
\end{proposition}

\begin{proof}
Recall that the  last two maps in the definition of the intersection form $Q_W$, \eqref{equ:defqw}, are isomorphisms.
On the other hand the first map fits into the long exact sequence
\[ \ba{ccccccccccc} &H_2(\partial W;\zt)&\to &H_2(W;\zt)&\to& H_2(W,\partial W;\zt)&\to\\
\to & H_1(\partial W;\zt)&\to& H_1(W;\zt) &\to& \ea \]
In our case $\partial W=M(K)\sqcup M(J)$, it thus follows that
\[ H_i(\partial W;\zt)=H_i(M(K);\zt)\oplus H_i(M(J);\zt) \mbox{ for $i=1,2$,}\]
which is $\zt$--torsion. Since $H_2(W,\partial W;\zt)$ is free and since $H_1(W;\zt)=0$ we now see that the above long exact sequence descends to  the following short exact sequence:
\[ \ba{rcl} 0&\to& \FH_2(W;\zt)\to H_2(W,\partial W;\zt)\to\\
 H_1(M(K);\zt)\oplus H_1(M(J);\zt)&\to& 0.\ea  \]
Let $A$ be a matrix representing $Q_W$ for a basis for $\FH_2(W;\zt)$.
It follows from the definition of $Q_W$ that the matrix  $A$ also represents the map $ \FH_2(W;\zt)\to H_2(W,\partial W;\zt)$ for some appropriate bases. We thus see that
$A$ is a presentation matrix for the $\zt$--module
\[ H_1(M(K);\zt)\oplus H_1(M(J);\zt),\]
which by the definition of the Alexander polynomials implies that
\[\det(A)\doteq\Delta_K(t)\cdot \Delta_J(t).\]
\end{proof}

%In the next sections, given a knot $K$ we want to construct a $\zt$--cobordism between $K$ and another knot $J$  such that $J$ is an Alexander polynomial one knot.
%Note that this is the case precisely when $\det(Q_W)=\Delta_K(t)$.

%===========================================
\subsection{Surgeries and intersection forms}\label{section:attach}

Let $M$ be a 3--manifold which is a homology $S^1\times S^2$.  Let $(c_1,\eps_1),\dots,(c_n,\eps_n)$ be framed oriented curves in $M$ with the following properties:
\bn
\item the framings are either $-1$ or $1$,
\item $c_1,\dots,c_n$ are homologically trivial in $M$.
\en
%Note that the  3-manifold which is obtained by performing $\eps_i$ surgery on $c_i$ for $i=1,\dots,n$ is again the 3--sphere.
%We denote by $J$ the resulting knot.
We then consider  the 4--manifold $W$ which is given by attaching 2--handles $h_1,\dots,h_n$ with framings $\eps_1,\dots,\eps_n$ to $M\times [0,1]$ along 
$c_1\times\{1\},\dots,c_n\times\{1\}\subset M\times\{1\}$. We identify $M$ with $M\times \{0\}$ and we denote by $M'$ the other boundary component of $W$.

It follows from (2) that   $H_1(W;\Z)=\Z$ and that the maps $H_1(M;\Z)\to H_1(W;\Z)$ and $H_1(M';\Z)\to H_1(W;\Z)$ are isomorphisms.
It furthermore follows from (2) that $c_1,\dots,c_n$  define elements of $H_1(M;\zt)$, which are well--defined up to a power of $t$.
It is straightforward to see that
\[ H_1(W;\zt)\cong H_1(M;\zt)/(c_1,\dots,c_n).\]

Next result is the second homological ingredient needed in the proof of Theorem~\ref{mainthm}.

\begin{proposition}\label{prop:detqw2}
 If $c_1,\dots,c_n$ generate $H_1(M(K);\zt)$, then $W$ is a $\zt$--cobordism between $M$ and $M'$, and
\[ \det(Q_W)\doteq\det\left(\{\lkt(c_i,c_j)\}_{ij}\right)\cdot \Delta_M(t)^2.\]
\end{proposition}

\begin{proof}
Throughout the proof we write $\Delta=\Delta_M(t)$.
It follows from the definitions and the discussion preceding the lemma that $W$ is indeed a $\zt$--cobordism between $M$ and $M'$.
We  consider the short exact sequence \eqref{exactsequence}
\[ 0\to \FH_2(W;\zt)\xrightarrow{\iota} H_2(W,M;\zt)\to H_1(M;\zt)\to 0.\]
It is clear that the cores of the 2--handles $h_1,\dots,h_n$ give rise to a generating set for $H_2(W,M;\zt)$.
By a slight abuse of notation we denote the cores of the 2--handles by $h_1,\dots,h_n$ as well.
Note that each $h_i$ then naturally defines an element $[h_i]\in H_2(W,M;\zt)$. 
By Lemma~\ref{annihilates},
there exist $k_1,\dots,k_n\in \FH_2(W;\zt)$, such that $\iota(k_i)=\Delta \cdot [h_i]\in H_2(W,M;\zt)$, $i=1,\dots,n$.

\begin{lemma}
\[ k_i\cdot k_j=\Delta^2\cdot \lkt(c_i,c_j).\]
\end{lemma}

%We now think of $W$ as equipped with a collar $M\times [0,1]$, where we identify the boundary component $M$ of $W$ with $M\times 1$.

\begin{proof}
  We denote the infinite cyclic covers of $M$ and $X=X(K)$ by
  $\wti{M}$ and $\wti{X}$. By Lemma~\ref{annihilates}
  we can find surfaces
  $F_1,\dots,F_n$ in $\wti{M}$ such that $\partial F_i=\Delta c_i$.
  We can arrange the surfaces such that $F_i$ and $t^kc_j$ are in
  general position for any $i,j,k$.

  We first consider the case $i\ne j$.  We then consider the surface
  \[ T_i:=\Delta\cdot h_i \,\,\cup \,\,(\Delta\cdot c_i\times
  [0,\frac{1}{2}])\,\,\cup \,\,(F_i\times \frac{1}{2}) \] in $\wti{W}$
  where we think of $\Delta\cdot h_i$ and $\Delta\cdot c_i$ as a
  disjoint union of appropriate translates of the surface $h_i$
  respectively the curve $c_i$.  Note that the surface $T_i$ is closed
  and the image of $[T_i]$ in $H_2(W,M;\zt)$ is the same as the image
  of $\Delta[h_i]$ in $H_2(W,M;\zt)$.  Since
  $\FH_2(W;\zt)\xrightarrow{\iota} H_2(W,M;\zt)$ is injective it now
  follows that $T_i$ represents the class $k_i$.  Similarly we
  consider the surface
  \[ T_j:=\Delta\cdot h_j \,\,\cup \,\,(\Delta\cdot c_j\times
  [0,1])\,\,\cup \,\,(F_j\times 1), \] where $F_j$ is a surface in
  $\wti{M}$ which has boundary $\Delta\cdot c_j$.  Note that the
  surface $T_j$ is closed and represents the class $k_j$.

  We can thus use the surfaces $T_i$ and $T_j$ to calculate $k_i\cdot
  k_j$. But it is clear from the definitions that
  \[ T_i\cdot T_j=(\Delta F_i\times \frac{1}{2}) \cdot (c_j\times
  \frac{1}{2}),\] but this clearly equals $\Delta\cdot (F_i\cdot c_j)=\Delta^2\cdot \lkt(c_i,c_j)$.

  The case $i=j$ can be proved completely analogously by constructing
  an appropriate surface $T_i'$ using the longitude of $c_i$ with
  framing $\eps_i$ which connects up with the core of the 2--handle
  which we had attached to $c_i$ with framing $\eps_i$. We leave the
  details to the reader.  This concludes the proof of the lemma.
\end{proof}

\begin{lemma}
The order of the $\zt$--module
\[ \FH_2(W;\zt)/(k_1,\dots,k_n)\]
equals $\Delta^{n-1}$.
\end{lemma}

\begin{proof}
It follows from  \eqref{exactsequence} and from the definitions that we have the following commutative diagram of maps where the horizontal sequences are exact:
\[ \xymatrix{ 
0\ar[r]\ar[d]& \bigoplus_{i=1}^n k_i\zt\ar[r]^\iota\ar[d] &  \bigoplus_{i=1}^n \Delta h_i\ar[d]\zt\ar[r] &0\ar[d]&&\\ 
0\ar[r]& \FH_2(W;\zt)\ar[r]^{\iota} &H_2(W,M;\zt)\ar[r]& H_1(M;\zt)\ar[r]& 0.}\]
It then follows that the following sequence of maps
  \[\ba{rcl} 0&\to& \FH_2(W;\zt)/(k_1,\dots,k_n)\to H_2(W,M;\zt)/(\Delta h_1,\dots,\Delta h_n)\\
  &\to& H_1(M;\zt)\to 0.\ea \] 
  is well--defined and  exact.  By the
  multiplicativity of orders (see Lemma \ref{lem:orders}) it follows
  that
  \[ \ba{ll} \order(H_2(W,M;\zt)/(\Delta h_1,\dots,\Delta h_n))\\
  \hspace{4cm} \doteq \order(\FH_2(W;\zt)/(k_1,\dots,k_n))\,\,\cdot
  \,\,\order(H_1(M;\zt)).\ea \] But the order on the left is clearly
  $\Delta^n$ and the order of $H_1(M;\zt)$ equals $\Delta$ by the
  definition of $\Delta$.  This concludes the proof of the lemma.
\end{proof}

Using Lemma \ref{lem:detqw} we now see that
\[ \ba{rcl} \det(Q_W)&\doteq&\det(\{\Delta^2\lkt(c_i,c_j)\}_{ij})\cdot \Delta^{-2(n-1)}=\det(\{\lkt(c_i,c_j)\}_{ij})\cdot \Delta^{2n}\cdot \Delta^{-2(n-1)}\\
&=&\det(\{\lkt(c_i,c_j)\}_{ij})\cdot \Delta^{2}.\ea \]
\end{proof}

%===========================================
\section{The  main theorem}\label{section:proof}

%===========================================
\subsection{Statement of the main theorem}

In this section we will state a slightly stronger version  of our main theorem.
In order to state the theorem we first have to  recall the following definition: A crossing change is a \emph{positive crossing change} if it turns a negative crossing into a positive crossing. Otherwise we refer to the crossing change as a \emph{negative crossing change}.

\smallskip
\begin{pspicture}(-3,-0.5)(6,2.2)
\rput(4,0){\psline(-4.5,0)(-3.75,0.75)\psline[arrowsize=6pt]{->}(-3.25,1.25)(-2.5,2)
\psline[arrowsize=6pt]{->}(-4.5,2)(-2.5,0)
\rput(-3.5,0){\psscalebox{1.25}{``$-$''}}
\psline(3,2)(3.75,1.25)\psline[arrowsize=6pt]{->}(4.25,0.75)(5,0)
\psline[arrowsize=6pt]{->}(3,0)(5,2)
\rput(4,0){\psscalebox{1.25}{``$+$''}}
\psline[linestyle=dashed,arrowsize=3pt]{->}(-2,1.1)(2,1.1)\rput(0,1.35){\psscalebox{0.8}{positive crossing change}}
\psline[linestyle=dashed,arrowsize=3pt]{<-}(-2,0.9)(2,0.9)\rput(0,0.65){\psscalebox{0.8}{negative crossing change}}
}
\end{pspicture}

The following theorem is now our main result, it clearly implies  Theorem \ref{mainthm} from the introduction.

\begin{theorem}\label{mainthm2}
Let $K$ be a knot and let $A=A(t)$ be an $n\times n$--matrix over $\zt$ such that $Bl(K)\cong \l(A)$ and such that $A(1)$ is diagonalizable over $\Z$.
We denote the number of positive eigenvalues of $A(1)$ by $n_+$ and we denote the number of negative eigenvalues by $n_-$.
Then $K$ can be turned into a knot with Alexander polynomial one using $n_+$ negative crossing changes and $n_-$ positive crossing changes.
\end{theorem}

There are two ingredients in the proof of Theorem. The homological part was given in Proposition~\ref{prop:deltakj} and Proposition~\ref{prop:detqw2}.
The main topological tool 
will be Lemma \ref{lem:maintechnicallem} which we will state in the following section.

\begin{remark}
 The theorem applies also to knots in $\Z$--homology sphere. In general, 
such a knot can not be unknotted using `crossing changes' (i.e. using surgeries along curves which bound nice disks) since the knot might not even be null-homotopic. But any knot can be turned into Alexander polynomial one knots, using $n(K)$ unknotting moves.
\end{remark}

%===========================================
\subsection{The  main technical lemma}\label{section:statementtechnical}

In order to state our main technical lemma we need a few more definitions:

\begin{definition}
Let $K\subset S^3$ be a knot.
A (based) disk $D\subset S^3$ is called \emph{nice} if the disk is embedded (that is the unbased disk is embedded), 
if it intersects $K$ transversely and if it intersects $K$ exactly twice
with opposite signs.
%A curve in $S^3$ is called \emph{nice} if it bounds a nice disk. Similarly we defined nice based disks and nice based curves.
\end{definition}

\begin{definition}
Let $D,D'$ be embedded disks in $S^3$. We say that the disk $D$ \emph{precedes} the disk $D'$ if
$D'$ and $D$ intersect transversely and if $D'\cap \partial D=\emptyset$.
\end{definition}

As an example, consider the disks in Figure \ref{fig:bluegreen},
\begin{figure}
\begin{pspicture}(-5,-2.5)(5,2.5)
\newrgbcolor{blue1}{0.0000 0.0000 0.500}
\newrgbcolor{blue2}{0.8 0.8 1.0}
\newrgbcolor{green1}{0.0 0.5 0.0}
\newrgbcolor{green2}{0.8 1.0 0.8}
\psellipse[linecolor=green1,fillcolor=green2,linewidth=2pt,fillstyle=solid](0,0)(0.6,2.0) % green disk
\psellipse[linecolor=white,fillcolor=white,linewidth=2pt,fillstyle=solid](0,0)(2.2,0.7) % clean the neighbourhood
\psellipse[linecolor=blue1,fillcolor=blue2,linewidth=2pt,fillstyle=vlines,hatchcolor=blue,hatchwidth=0.01,hatchsep=0.05](0,0)(2.0,0.6) % blue disk
\psellipticarc[linecolor=white,fillstyle=none,linewidth=2pt](0,0)(2.0,0.6){70}{110} % erase part of it
\begin{psclip}{ \psellipse[linestyle=none](0,0.2)(2.0,0.6)}
\psellipse[linecolor=green1,fillcolor=green2,linewidth=2pt,fillstyle=solid](0,0)(0.6,2.0) % green disk. we draw it again.
\end{psclip}
\psellipticarc[linecolor=black,linewidth=1.5pt, linestyle=dashed](0,0)(1.5,0.43){220}{320}
\end{pspicture}
\caption{The blue (the dashed one) and green  disks.}
\label{fig:bluegreen}
\end{figure}
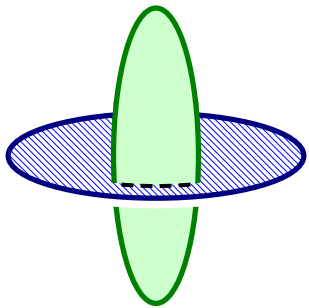
then the blue (dashed) disk precedes the green (solid) disk, but not vice versa.

We can now state our main technical lemma. It will be proved in Section~\ref{proofoflemma}.

\begin{lemma}\label{lem:maintechnicallem}
Let $K$ be a knot and let $x_1,\dots,x_n$ be elements in $H_1(X(K);\zt)$.
Let $p_{ij}(t)\in \zt$, $i,j\in \{1,\dots,n\}$ be such that
\[ Bl(x_i,x_j)= \frac{p_{ij}(t)}{\Delta_K(t)} \in \Q(t)/\zt \mbox{ and }p_{ij}(t)=p_{ji}(t^{-1})\]
for any $i$ and $j$. Then  there exists an ordered set  $\{D_1,\dots,D_n\}$ of
 based nice disks with the following properties:
\bn
%\item the disks $D_1,\dots,D_n$ are disjoint,
\item for $i<j$ the disk $D_i$ precedes $D_j$,
\item for any $i$ the based curve $c_i:=\partial D_i$ represents $x_i$,
\item if for $i=1,\dots,n$ we equip $c_i=\partial D_i$ with the framing $p_{ii}(1)$, then
\[ \lkt(c_i,c_j)=\frac{p_{ij}(t)}{\Delta_K(t)}\in \Q(t),\]
for any $i$ and $j$.
\en
\end{lemma}

%===========================================
\subsection{Proof of Theorem \ref{mainthm2} assuming Lemma \ref{lem:maintechnicallem}}

We will now prove  Theorem \ref{mainthm2} using  Lemma \ref{lem:maintechnicallem}.
Let $K$ be a knot. We write $\Delta=\Delta_K(t)$.
 Let $A=A(t)$ be an $n\times n$--matrix over $\zt$ such that $Bl(K)\cong \l(A)$ and such that $A(1)$ is diagonalizable over $\Z$.
We denote the number of positive eigenvalues of $A(1)$ by $n_+$ and we denote the number of negative eigenvalues by $n_-$.

Note that since  $A(1)$ is diagonalizable over $\Z$ we can  find an invertible matrix $P$ over $\Z$ such that $PA(1)P^t$ is diagonal over $\Z$.
We can thus, without loss of generality assume, that $A(1)$ is diagonal.

The matrix $A(t)$ is in particular a presentation matrix for the Alexander module.
It follows that $\det(A(t))=\pm \Delta_K(t)$ and in particular $\det(A(1))=\pm 1$.
The entries on the diagonal of $A(1)$ are therefore either $+1$ or $-1$.
 We now denote by $\eps_1,\dots,\eps_n$ the diagonal entries.
Given $i,j\in \{1,\dots,n\}$ we  denote by $b_{ij}(t)\in \zt$ the polynomial which satisfies
\[ \mbox{$ij$--entry of $A(t)^{-1}$}=\frac{b_{ij}(t)}{\Delta}.\]

We denote by $e_1,\dots,e_n$ the canonical generating set of $\zt^n/A\zt^n$ and we denote by $x_1,\dots,x_n$ the images of $e_1,\dots,e_n$
under the isometry $\l(A)\to Bl(K)$.
By Lemma \ref{lem:maintechnicallem}
 there exists an ordered  set  $\{D_1,\dots,D_n\}$ of
 based nice disks with the following properties:
\bn
\item for any $i<j$ the disk $D_i$ precedes $D_j$,
\item for any $i$ the based curve $c_i:=\partial D_i$ represents $x_i$,
\item if for $i=1,\dots,n$ we equip $c_i=\partial D_i$ with the framing $b_{ii}(1)$, then
\[ \lkt(c_i,c_j)=\frac{b_{ij}(t)}{\Delta},\]
for any $i$ and $j$.
\en

We now consider the disk $D_1$. After an isotopy of $S^3$ we can assume that it is `standard' as in Figure \ref{fig:standard} on the left.
\begin{figure}[h]
\begin{tabular}{cc}
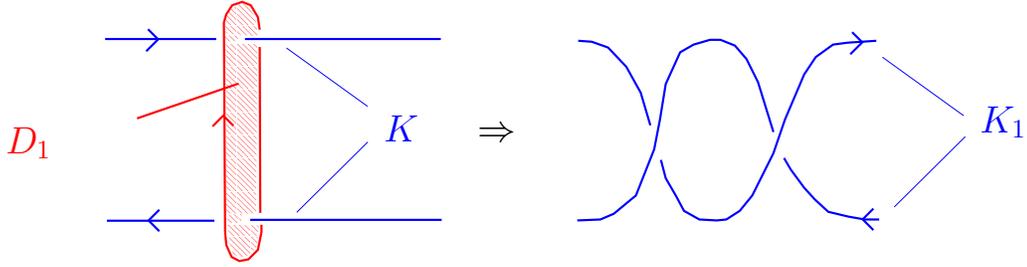
\end{tabular}
\caption{A nice disk in standard position and the result of adding a full $+1$--twist along the disk.}
\label{fig:standard}
\end{figure}
We now perform $\eps_1$--surgery on the unknot $c_1=\partial D_1$.
The resulting 3--manifold is again $S^3$.
Furthermore the knot $K_1$,
which is defined as the image of $K$ in the surgery $S^3$, is obtained from $K_0:=K$ through adding a full $\eps_1$--twist along the disk
(see Figure \ref{fig:standard}). Adding a full $\eps_1$--twist
corresponds to  a $(-\eps_1)$--crossing change in an appropriate diagram of $K$.
%Let $\eps\in \{-1,1\}$. We  modify the diagram by introducing an $\eps$ and a $-\eps$--crossing to the left of the disk.
%\begin{figure}[h]
%\begin{tabular}{ccc}
%\includegraphics[scale=0.22]{addtwist1.eps}&\includegraphics[scale=0.22]{addtwist2.eps}&\includegraphics[scale=0.22]{addtwist3.eps}
%\end{tabular}
%\caption{Adding a canceling pair of crossings and the effect of twisting along $D$.}
%\label{fig:addcrossings}
%\end{figure}
%If we now perform a full $\eps$--twist along $D$ for each $i$, then the resulting knot can also be obtained by
%performing an $\eps$--crossing change. (We refer to Figure \ref{fig:addcrossings} for an illustration.)
The fact that $D_1$ precedes $D_2,\dots,D_n$  implies
that the disks $D_2,\dots,D_n$ are `unaffected' by the surgery, in particular for $j=2,\ldots,n$, $\partial D_j$ is again an unknot
and for $2\le i<j$, $D_i$ precedes $D_j$.
We can therefore iterate this process, and perform $\eps_i$--surgery along  the unknots $c_i=\partial D_i$ for $i=2,\dots,n$.
As given $i<j$ the disk $D_i$ precedes $D_j$, the consecutive 
surgeries do not affect the remaining disks, in particular at each step the remaining curves are unknots in the 3--sphere.

We denote the resulting knots by $K_2,\dots,K_n$. As above, for each $i=2,\dots,n$
the knot $K_i$ is  obtained from $K_{i-1}$ by doing an $\eps_i$--crossing change.
In particular  $K=K_0$ can be turned into the knot $J:=K_n$ using $n_+$ negative crossing changes and $n_-$ positive crossing changes.
It remains to show that $\Delta_J(t)=1$.

For $i=0,\dots,n-1$ we now denote by $W_i$ the result of adding 2--handles along $c_{i+1}$ to $M(K_{i})\times [0,1]$ with framing $\eps_{i+1}$. 
Adding a 2--handle gives a cobordism between the original manifold and the surgered 3--manifold.
In particular we see that  $\partial W_i=-M(K_{i})\sqcup M(K_{i+1})$.
We can also add all the 2--handles simultaneously along $c_1,\dots,c_n$ with framings $\eps_1,\dots,\eps_n$ and we thus obtain a 4--manifold $W$
which is diffeomorphic to the union $W_1,\dots,W_n$ along the corresponding boundaries.
Note that $\partial W=-M(K)\sqcup M(J)$.
By the discussion of Section \ref{section:attach} the manifold $W$ has furthermore the following properties:
\bn
\item  $H_1(W;\Z)=\Z$,
\item $H_1(M(K);\Z)\to H_1(W;\Z)$ and $H_1(M(J);\Z)\to H_1(W;\Z)$ are isomorphisms,
\item $H_1(W;\zt)\cong H_1(M(K);\zt)/(c_1,\dots,c_n)=0$.
\en
Furthermore, by Proposition \ref{prop:detqw2} we see that
\[ \det(Q_W)\doteq \det\left(\{\lkt(c_i,c_j)\}_{ij}\right)\cdot \Delta^{2}\doteq \det(A(t)^{-1})\cdot \Delta^2\doteq \Delta^{-1}\cdot \Delta^2=\Delta.\]
It now follows from Proposition \ref{prop:deltakj} that the knot $J=K_n$ has trivial Alexander polynomial.
This concludes  the proof of Theorem \ref{mainthm2}, modulo the proof  of Lemma \ref{lem:maintechnicallem} which will be given in the next section.

%===========================================
\section{Proof of Lemma~\ref{lem:maintechnicallem}}\label{proofoflemma}

In this section we shall prove Lemma~\ref{lem:maintechnicallem}. The proof is given in a couple of steps. First, we find pairwise
disjoint nice disks $D_1,\ldots,D_n$, with $c_j=\partial D_j$, such that for any $i,j=1,\ldots,n$ we have
$Bl(c_i,c_j)=\frac{p_{ij}(t)}{\Delta(t)}\in\Q(t)/\zt$. This is an adaptation of Fogel's
argument \cite[p.~287]{Fo94} and is done in Section~\ref{sec:findingdisks}. The property that $Bl(c_j,c_j)=\frac{p_{ij}(t)}{\Delta(t)}\in\Q(t)/\zt$ is weaker
that $\lkt(c_i,c_j)=\frac{p_{ij}(t)}{\Delta(t)}\in\Q(t)$, it only means that $\lkt(c_i,c_j)-\frac{p_{ij}(t)}{\Delta(t)}$ is an element of $\zt$.

To ensure that $\lkt(c_i,c_j)-\frac{p_{ij}(t)}{\Delta(t)}=0$ we need to perform several moves on the disks. We introduce four types of moves in 
Section~\ref{sec:typeRF} and one type in Section~\ref{section:proofii}. These
moves potentially introduce intersections among disks $D_1,\ldots,D_n$, therefore an analysis must be careful and take into account
the ordering of disks. In our prof we perform only the moves that preserve the ordering  of the disks. The details are given in Section~\ref{sec:happyending}.

%==============================================
\subsection{Finding nice based disks}\label{sec:findingdisks}
In this section we prove the following lemma.

\begin{lemma} \label{lem:fogel}
Let $K$ be a knot and let $x_1,\dots,x_n\in H_1(X(K);\zt)$.
Then there exist $n$ disjoint nice based disks $D_1,\dots,D_n$ such that
for $i=1,\dots,n$ the curve  $c_i:=\partial D_i$ represents $x_i$.
\end{lemma}

This lemma is a slight generalization of a result by Fogel
 \cite[p.~287]{Fo94}. The proof we give is also basically due to Fogel.

\begin{proof}
Let $x_1,\dots,x_n\in H_1(X(K);\zt)$. The multiplication by $t-1$ is an isomorphism of $H_1(X(K);\zt)$ (see e.g. \cite{Lev77}). We can therefore find $y_1,\dots,y_n\in H_1(X(K);\zt)$
such that $(t-1)y_i=x_i, i=1,\dots,n$. We now represent $y_1,\dots,y_n$ by disjoint based curves $d_1,\dots,d_n$.
(By doing crossing changes on the curves $d_1,\ldots,d_n$, we can without loss of generality assume that the unbased curves are unknotted in $S^3$, this justifies the illustration below, but is not necessary for the argument.)
We also pick disjoint embedded oriented disks $S_1,\dots,S_n$ with the following properties:
\bn
\item for $i=1,\dots,n$ the disk $S_i$ intersects $K$ precisely once with positive intersection number,
\item for $i=1,\dots,n$ the curve $m_i:=\partial S_i$ intersects $d_i$ in precisely one point,
\item   for $i\ne j$ the curves $m_i$ and $d_j$ are disjoint.
\en
We refer to Figure \ref{fig:nicedisks} for a schematic picture.
\begin{figure}[h]
\begin{tabular}{cc}
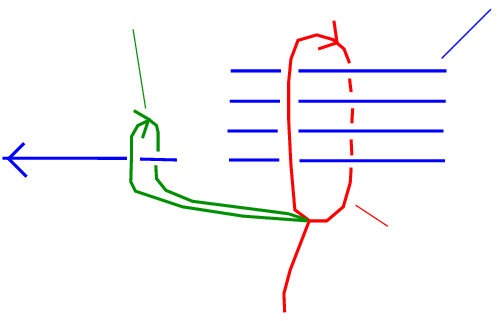& 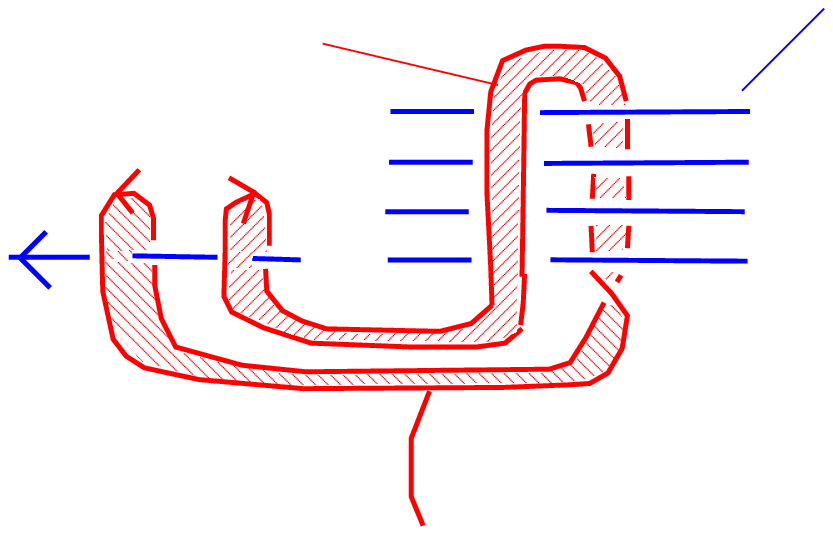
\end{tabular}
\caption{Construction of nice disks.}
\label{fig:nicedisks}
\end{figure}
Now note that for each $i$ the unbased curve $m_id_im_i^{-1}d_i^{-1}$ bounds a nice disk $D_i$ which can be placed in a small  neighborhood around the disk $S_i$ and the unbased curve $d_i$. (We again refer to Figure \ref{fig:nicedisks} for a schematic picture.) By construction the  disks
$D_1,\dots,D_n$ are disjoint. On the other hand, since $m_i$ is a meridian we see that
$m_id_im_i^{-1}d_i^{-1}=(m_id_im_i^{-1})\cdot d_i^{-1}$ represents $ty_i-y_i=(t-1)y_i=x_i$ in the Alexander module.
If we equip $D_1,\dots,D_n$ with the basings of the based curves $d_1,\dots,d_n$ we thus obtain the required based disks.
%It is clear that if we use appropriate meridians $m_i$ and appropriate push--offs for $m_i^{-1}$, then we can arrange that the disks are properly arranged.
\end{proof}

%==============================================
\subsection{Properly arranged disks}

The following discussion will be essential in the remainder of the proof.

\begin{definition}\label{properly}
Let $K\subset S^3$ be a knot and let $D_1,\dots,D_n$ be  nice based disks.
We say that they are \emph{properly arranged} if the following conditions hold:
\bn
\item the segment $S:=[0,1]\times 0\times 0\subset \R^3\subset S^3$ is part of the knot $K$,
and the orientation of $K$ agrees with the canonical orientation on that segment,
\item all intersection points of the disks with the knot $K$ lie on $S$,
\item for $i<j$ the disk $D_i$ precedes the disk $D_j$.
\en
\end{definition}
\begin{remark}
If $D_1,\ldots,D_n$ are nice based disks that are disjoint, then it is 
straightforward to see that  a segment $S\subset K$ exists which satisfies Conditions  (1) and (2) from Definiton~\ref{properly}.
\end{remark}

 Note that if the disks $D_1,\dots,D_n$ are properly arranged then  we can find a tubular neighborhood of the segment $S$ of the knot $K$ which is isotopic to the  picture shown in Figure \ref{fig:properlyarranged}.
We call such a neighborhood of $S$ a \emph{standard segment}. We refer to each of the $2n$ components of the disks
as a \emph{piece}. The orientation on the disks endows each piece with an orientation, which we refer to as positive or negative
depending on the intersection with the oriented $S$. Finally each cube  in $S$ which contains precisely two pieces
is called a \emph{subsegment}.
In the following we will furthermore use the expressions `adjacent pieces' and `piece to the left' and `piece to the right' with the obvious meanings.

% Subsequently we will show various moves in subsegments.

%We call such a neighborhood a \emph{proper cylinder} of $(K,D_1,\dots,D_n)$.

We henceforth equip the set of points $S$ with the canonical ordering coming from the ordering on the interval $[0,1]$.
If the  disjoint nice disks $D_1,\dots,D_n$ are properly arranged then the intersection points are of the form
$z_1< z_2< \dots <z_{2n}$. Given $i\in \{1,\dots,2n\}$ we denote by $\s(i)\in \{1,\dots,n\}$ the integer which has the property that
the disk $D_{\s(i)}$ intersects $S$ in the point $z_i$.
We refer to the ordered set
\[ \{\s(1),\dots,\s(2n)\} \]
as the \emph{arrangement} of the properly arranged disks $D_1,\dots,D_n$.
\begin{figure}[h]
\begin{center}
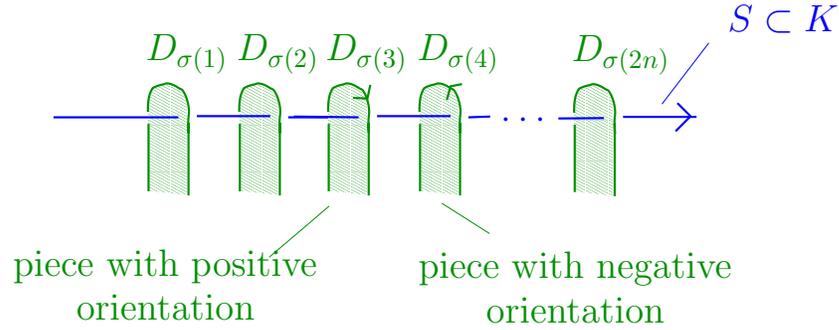
\caption{Properly arranged disks.}\label{fig:properlyarranged}
\end{center}
\end{figure}
We refer to Figure \ref{fig:properlyarranged} for an illustration.

\subsection{Type $R$ and $F$ moves}\label{sec:typeRF}

Given properly arranged nice disks $D_1,\dots,D_n$ we consider the following  local moves
which produce new sets of properly arranged nice disks $D_1',\dots,D_n'$.
The subsequent figures show moves on sets of properly arranged disks which take place in subsegments, in particular  no other disks and no 
basings are allowed in these subsets of $S^3$.
%We write $c_i'=\partial D_i'$, $i=1,\dots,n$.

\smallskip
\noindent
If $j<i$, then a \textbf{type $R_1$ move} consists of the change as drawn on Figure~\ref{fig:swap},
i.e. we push the disk $D_j$ `on the right' over the disk $D_i$ `on the left'.
\begin{figure}
%\Rmove
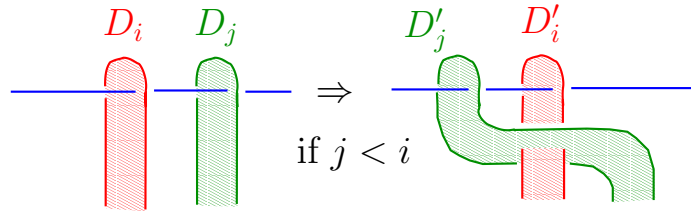
\caption{A type $R_1$ move.}\label{fig:swap}
\end{figure}
Note that the isotopy types of the boundary curves are unchanged, so the twisted linking numbers of the new boundary curves
agree with the twisted linking numbers of the old boundary curves.
The resulting disk $D_j'$ precedes $D_i'$,  because $D_j$ precedes $D_i$.

\smallskip
\noindent If $j>i$, then a \textbf{type $R_2$ move} consists of the move  shown in Figure \ref{fig:swap2},
\begin{figure}
%\Rmove
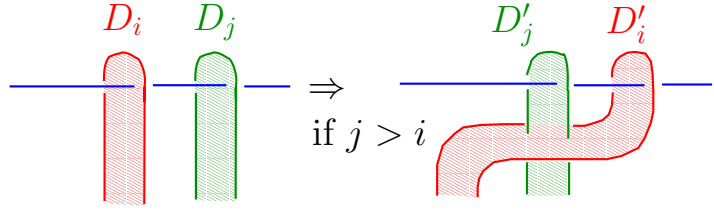
\caption{A type $R_2$ move.}\label{fig:swap2}
\end{figure}
which is almost of the same form as the type $R_1$ move,
except that we now push the disk $D_i$ `on the left' over the disk $D_j$ `on the right'.
Note that  $D_1',\dots,D_n'$ are again properly arranged.

\smallskip
\noindent A \textbf{type $F_1$ move} consists of applying the move shown in Figure \ref{fig:typef} to two adjacent pieces with opposite orientations.
%It changes two disks $D_i$ and $D_j$ locally, and leaves all other disks unchanged.
%The first move is in fact just an isotopy of the disks $D_i$ and $D_j$ (this will be shown in the next section), the second move consists of pushing
%the disk $D_j$  over the disk $D_i$, this is an isotopy of the boundary curves but of course not of the disks.
%Finally the last move consists of pushing one strand of $\partial D_j$ over $\partial D_i$.
%This removes the intersection of the disks but changes the isotopy class of $\partial D_j$.
\begin{figure}
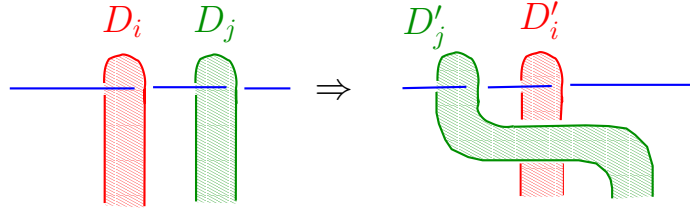
%\Fmove
\caption{Type $F_1$ move.}\label{fig:typef}
\end{figure}

\smallskip
\noindent A \textbf{type $F_2$ move} consists of applying the move shown in Figure \ref{fig:typefii} to two adjacent pieces with opposite orientations.
%It changes two disks $D_i$ and $D_j$ locally, and leaves all other disks unchanged.
%The first move is in fact just an isotopy of the disks $D_i$ and $D_j$ (this will be shown in the next section), the second move consists of pushing
%the disk $D_j$  over the disk $D_i$, this is an isotopy of the boundary curves but of course not of the disks.
%Finally the last move consists of pushing one strand of $\partial D_j$ over $\partial D_i$.
%This removes the intersection of the disks but changes the isotopy class of $\partial D_j$.
\begin{figure}
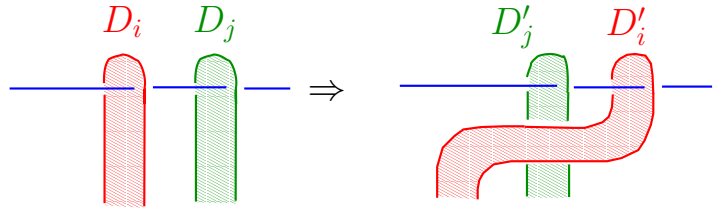
%\Fmove
\caption{Type $F_2$ move.}\label{fig:typefii}
\end{figure}
\begin{remark}
One could define $F$ moves for two adjacent pieces with the same orientation, but we will not need that.
\end{remark}

We denote  by $D_1',\dots,D_n'$ the disks resulting from applying a type $F_1$ move or a type $F_2$ move
to disks $D_1,\dots,D_n$. We write $c_l:=\partial D_l$ and $c_l':=\partial D_l'$ for $l=1,\dots,n$.
Note that neither move  creates any new intersections between the disks.
 In particular $D_1',\dots,D_n'$ are again properly arranged.
 On the other hand the isotopy type of the boundary curves changes. To state how the twisted linking numbers change we consider
 the curve $d$ which is given by concatenation of the following paths:
 \bn
 \item a path  from the base point $*$ along the based curve $c_j=\partial D_j$ to a point $P_j$ on the  piece of $D_j$ involved in the type $F$ moves,
 \item a horizontal path to the corresponding point $P_i$ on $c_i=\partial D_i$,
 \item a path from the point $P_i$ on $c_i$ to the base point $*$ along the based curve $c_i$.
 \en
 We refer to Figure \ref{fig:curved}  for an illustration.
 \begin{figure}
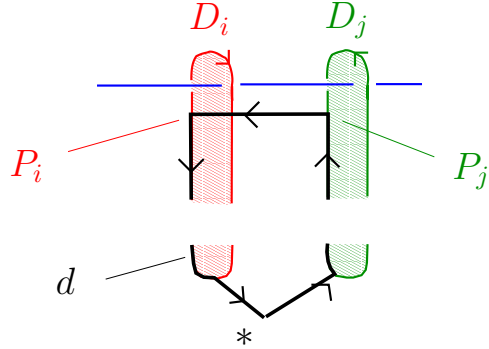
%\Fmove
\caption{Curve $d$ for different orientations of the pieces.}\label{fig:curved}
\end{figure}
We then denote  by
\begin{equation}\label{eq:k}
k:=k(D_i,D_j)
\end{equation}
the image of $d$ under the epimorphism $\pi_1(X_K)\to \Z$ given by sending the oriented meridian of $K$ to $1$. It is straightforward to see that $k$ is independent of the choice of $P_j$ made.

\begin{lemma}\label{lem:cchange}
For any $r,s$ with $\{r,s\}\ne \{i,j\}$ we have
\[  \lkt(c_r',c_s')=\lkt(c_r,c_s), \]
furthermore
\bn
\item if $i\ne j$, then
\begin{equation}\label{equ:lcij}
  \lkt(c_i',c_j')=\lkt(c_i,c_j)+\eps t^k(t^\eta-1) \mbox{ and } \lkt(c_j',c_i')=\lkt(c_j,c_i)+\eps t^{-k}(t^{-\eta}-1),
\end{equation}
\item if $i=j$,  then
\begin{equation}\label{equ:lcii}
\lkt(c_i',c_i')=\lkt(c_i,c_j)+ \eps t^k(t^\eta-1)+\eps t^{-k}(t^{-\eta}-1),
\end{equation}
\en
where $\eps=-1$ if we apply a type $F_1$ move and $\eps=1$ if we apply a type $F_2$ move,
furthermore $\eta=-1$ if the piece on the left  has positive orientation and $\eta=1$ if the piece on the left  has negative orientation.
\end{lemma}

We will first consider the case  of a type $F_1$ move such that the piece on the left has positive  orientation.\\

\smallskip
\noindent\textbf{Case 1. $i\neq j$.} 
It is clear, that for $\{r,s\}\ne \{i,j\}$ we have
$ \lkt(c_k',c_l')=\lkt(c_k,c_l)$.
We will now show that
\[ \lkt(c_j',c_i')=\lkt(c_j,c_i)+ t^{k}(t-1).\]
The claim regarding $\lkt(c_i',c_j')$ then follows from the antisymmetry of the twisted linking number.

First recall that the twisted linking numbers only depend on isotopy invariants of the curves.
We can therefore ignore the disks and we can also first apply a type $R_1$ move, which is an isotopy.
We therefore have to compare the twisted linking numbers of  the  two sets of curves shown in Figure \ref{fig:typerf}.
\begin{figure}[h]
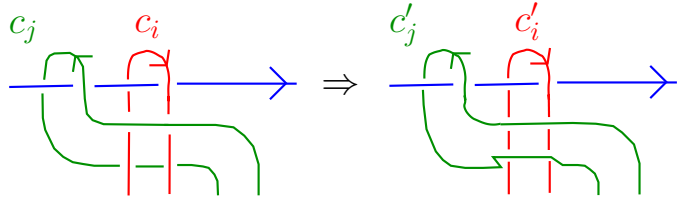
 \caption{Composition of the inverse of a type $R_1$ move and a type $F$ move.} \label{fig:typerf}
\end{figure}

We  pick a based immersed surface $F$ such that $\partial F=\Delta_K(t)\cdot c_j$.
In the subsegment  we can and will assume that the surface
$F$ is orthogonal to the plane  which contains the diagram and that it points `upwards'.
We now obtain a surface $F'$ with $\partial F'=\Delta_K(t)\cdot c_j'$ by  cutting out a small rectangle of $F$ around the modification.
The surfaces $F$ and $F'$ in the neighborhood of the modification are sketched in Figure \ref{fig:f}.
\begin{figure}[h]
\begin{tabular}{cc}
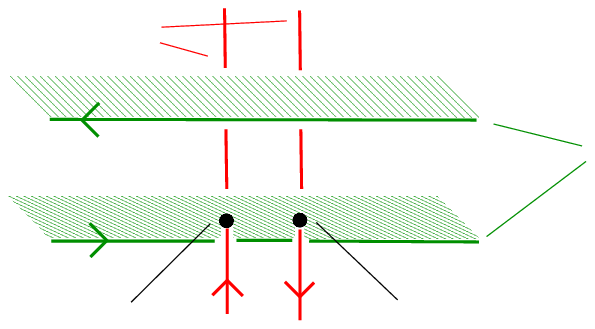&\hspace{1cm}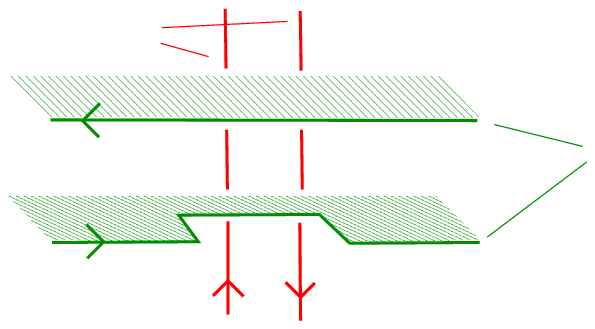
\end{tabular}
\caption{One sheet of $F$ respectively $F'$ in the subsegment
glued to $c_i$ as in Figure \ref{fig:typerf}. The surfaces go `vertically out of the plane' in the direction of the reader. On the left, the lower vertical sheet thus intersects $c_i$ in two points $P$ and $Q$. 
On the right, we pushed the surface across $c_i$, and thus removed the intersection points.}
\label{fig:f}
\end{figure}
Note that in Figure \ref{fig:f} we only show one  sheet of the surfaces $F$ and $F'$, in reality each sheet which is drawn should be considered $\Delta_K(t)$--times.

We are now interested in the difference between $F\cdot c_i$ and $F'\cdot c_i'$.
In the subsequent discussion we will continue with the notation in the definition of $F\cdot c_i$ (see Section \ref{section:basedcs}).
We consider the intersection points $P$ and $Q$ of $F$ and $c_i$ as shown in Figure \ref{fig:f} on the left.
It is clear that $\eps_P=-1$ and $\eps_Q=+1$.
It furthermore follows easily from the definitions (see also Figure \ref{fig:lpq})
\begin{figure}[h]
\begin{tabular}{cc}
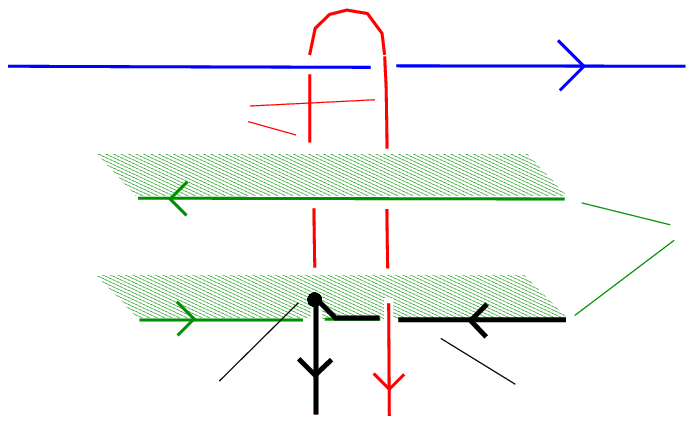&\hspace{1cm}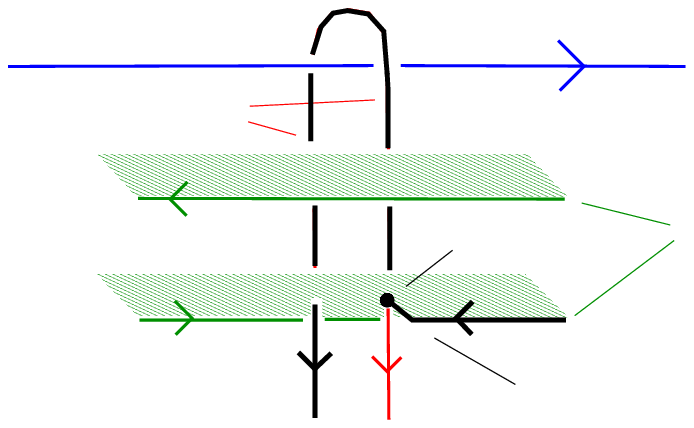
\end{tabular}
\caption{The curves $l_P$ and $l_Q$ in the definition of $F\cdot c$.
Here the sheets are again pointing outwards toward the reader. The upper sheet lies above $c_i$ and thus has no intersections with $c_i$, whereas the lower sheet intersects $c_i$ in two points $P$ and $Q$.}
\label{fig:lpq}
\end{figure}
 that
\[ \phi(l_P)=t^{k} \mbox{ and } \phi(l_Q)=t^{k-1}.\]
(The point is that in the definition of $\phi(l_P)$ the curve $l_P$ wraps around the knot once more in the negative direction.)
Now recall that $F$ and $F'$ consist of $\Delta_K(t)$ copies of the sheets indicated in the diagrams. It now follows that
\[ \ba{rcl} \lkt(c_j',c_i')&=& F'\cdot c_i'\\
&=& F\cdot c_i-\Delta_K(t)\cdot \eps_{P}\,\phi(l_P)-\Delta_K(t)\cdot \eps_{Q}\,\phi(l_Q)\\
&=&F\cdot c_i-\Delta_K(t)(t^{k-1}-t^{k})\\
&=&\lkt(c_j,c_i)-\Delta_K(t)t^k(t^{-1}-1).\ea \]
This concludes the proof in the case that $i\ne j$.\\

\smallskip
\noindent\textbf{Case 2. $i=j$.} We again pick a based immersed surface $F$ such that $\partial F=\Delta_K(t)\cdot c_i$.
In a neighborhood of the modification we can and will assume that the surface
$F$ is orthogonal to the plane  which contains the diagram and that it points `upwards'.
We again obtain a surface $F'$ with $\partial F'=\Delta_K(t)\cdot c_j'$ by  cutting out a small rectangle of $F$ around the modification.
The surfaces $F$ and $F'$ in the neighborhood of $c_i$ and the modification are sketched in Figure \ref{fig:f}.
\begin{figure}[h]
\begin{tabular}{cc}
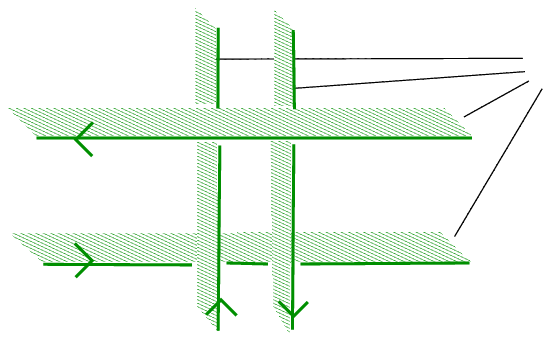&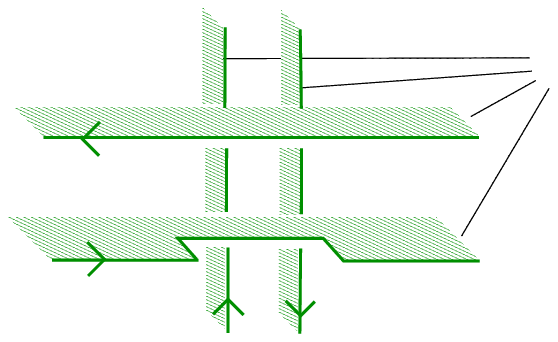
\end{tabular}
\caption{One sheet of $F$ respectively $F'$ in the subsegment.
The surfaces $F$ and $F'$ point vertically outwards towards the reader. They are indicated only in a small neighborhood of the curves and they have to  be extended in the direction of the reader beyond what is shown. In particular on the left the lower horizontal vertical sheet of $F$  intersects the two vertical sheets of $F$. On the other hand, on the right the two vertical sheets of $F'$  intersect the  lower horizontal sheet of $F'$.}
\label{fig:ffprim}
\end{figure}
Note that $F\cap c_i$ contains two intersection points, $P$ and $Q$, which do not appear in $F'\cap c_i'$, in turn $F'\cap c_i'$ contains two new intersection points, namely $P'$ and $Q'$.
We refer to Figure \ref{fig:i} for an illustration.
\begin{figure}[h]
\begin{tabular}{cc}
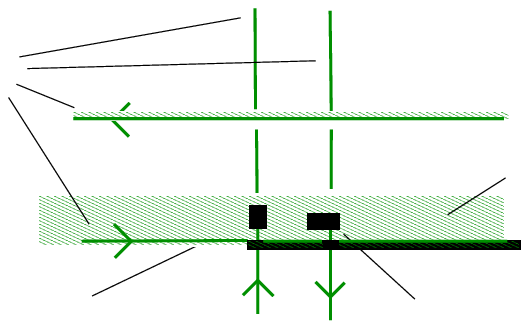&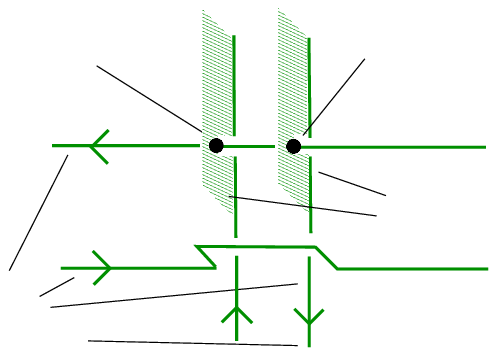
\end{tabular}
\caption{Extra intersection points of $F$ and $c_i$ respectively of $F'$ and $c'_i$. Here we show only parts of the surfaces $F$ and $F'$, the two parts are again meant to point outwards towards the reader.}
\label{fig:i}
\end{figure}
Note that in Figure \ref{fig:i} we now only indicate the parts of the sheets of $F$ and $F'$ which contain the extra intersection points.
% We thus do not show any of the sheets which have no affect on the calculation, e.g. we do not show the sheet cobounding the lower blue horizontal strand and we do not show the lower parts of the surface cobounding the green strands.
A careful consideration of the intersection points now shows that
\[\lkt(c_i',c_i')=\lkt(c_i,c_i)- t^k(t^{-1}-1)-t^{-k}(t-1).\]
We leave the details to the reader.

This concludes the proof of Lemma \ref{lem:cchange} in
 the case  of a type $F_1$ move such that the piece on the left has positive  orientation.
 It is straightforward to verify that the other cases of Lemma \ref{lem:cchange} can be proved completely analogously.
 We again leave the details to the reader.

\smallskip
%====================================
\subsection{The type $T(n)$ move}\label{section:proofii}

\noindent A \emph{type $T(n)$ move} consists of applying  the move shown in Figure \ref{fig:typet} to the based disk $D_i$.
\begin{figure}
\begin{tabular}{cc}
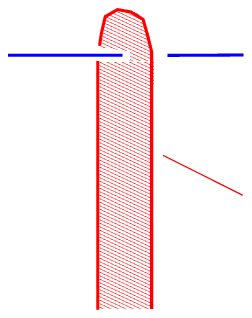&\hspace{0.3cm}%
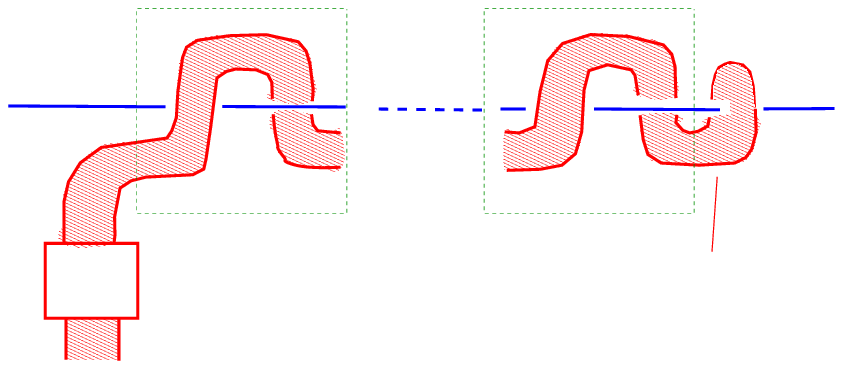
 \end{tabular}
%\Fmove
\caption{Type $T(n)$ move.}\label{fig:typet}
\end{figure}
This move is in fact an isotopy of the disk $D_i$ as will be shown later in Lemma \ref{lem:isotopy}.
In particular this move leaves all twisted linking numbers unchanged.
The move is important because it allows us to modify the term $k(D_i,D_j)$ which appears in the $F$--moves, see \eqref{eq:k}.
More precisely, suppose  we have two adjacent pieces of $D_i$ and $D_j$, with the piece corresponding to $D_i$ to the left.
Let $k\in \Z$ be the integer which is defined as in the discussion of the type $F$ moves.
If we first apply a type $T(n)$ move to $D_i$, then
\be \label{equ:kij} k(D_i',D_j)=k(D_i,D_j)+n.\ee

We will prove  the following lemma which shows that the type $T(n)$ move does not change the
isotopy type of the disk involved.

\begin{lemma}\label{lem:isotopy}
The two disks in Figure \ref{fig:isotopy} are isotopic relative to the boundary of the cube which contains the figures.
\begin{figure}[h]
\begin{tabular}{cc}
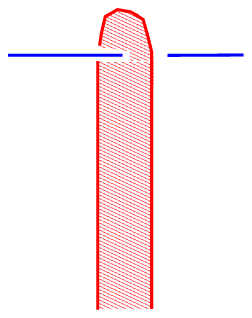&
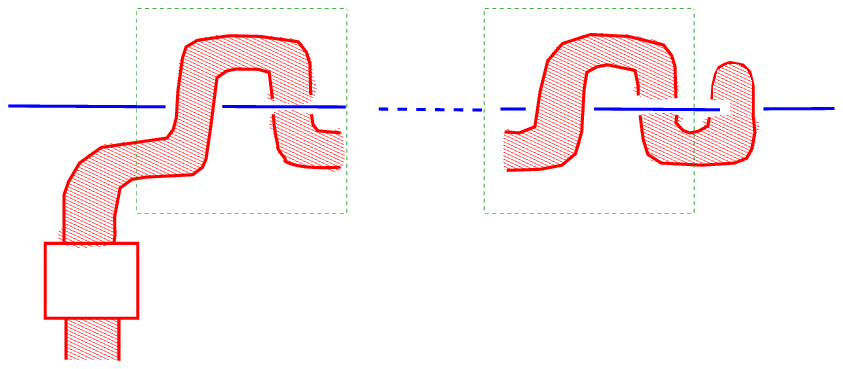
 \end{tabular}
 \caption{Isotopic disks in the statement  of Lemma \ref{lem:isotopy}.} \label{fig:isotopy}
\end{figure}
\end{lemma}

\begin{proof}
We first consider the set of isotopies (relative to the boundary of the cube) in Figures \ref{fig:iso1} and \ref{fig:iso2}.
\begin{figure}[h]
\begin{tabular}{cc}
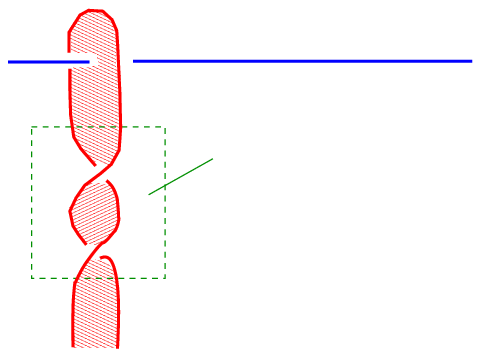&
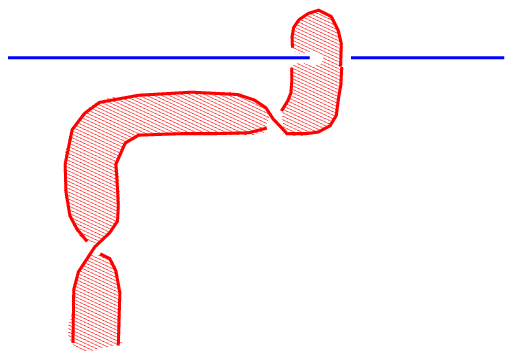
 \end{tabular}
 \caption{First isotopy in the proof of Lemma \ref{lem:isotopy}.} \label{fig:iso1}
\end{figure}
\begin{figure}[h]
\begin{tabular}{ccc}
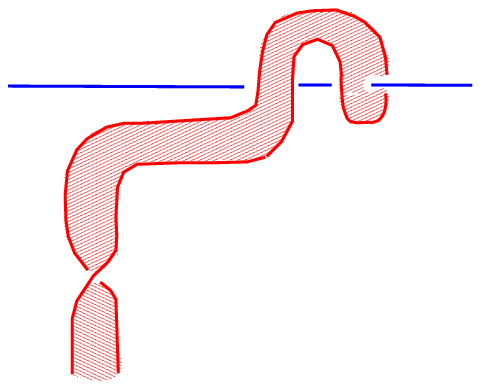&
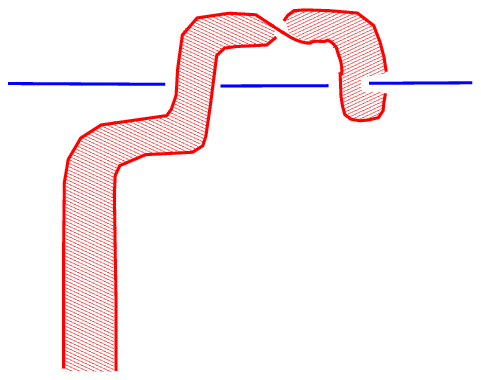&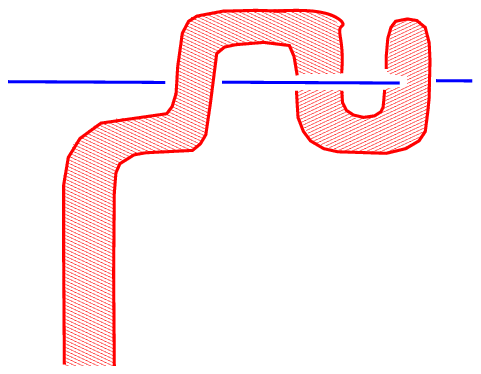
\end{tabular}
 \caption{Second set of isotopies in the proof of Lemma \ref{lem:isotopy}.} \label{fig:iso2}
\end{figure}
We then iterate  this process $k$ times. The lemma now follows from first adding a canceling pair of a full $k$ twist and a full $-k$ twist to the disk on the left hand side of Figure \ref{fig:isotopy}.
\end{proof}

\subsection{Proof of Lemma~\ref{lem:maintechnicallem}}\label{sec:happyending}

We are now in a position to prove Lemma \ref{lem:maintechnicallem}.

\begin{proof}[Proof of Lemma \ref{lem:maintechnicallem}]
Let $K$ be a knot and let $x_1,\dots,x_n$ be elements in $H_1(X(K);\zt)$.
Let $p_{ij}(t)\in \zt$, $i,j\in \{1,\dots,n\}$ be such that
\[ Bl(x_i,x_j)= \frac{p_{ij}(t)}{\Delta_K(t)} \in \Q(t)/\zt \mbox{ and } p_{ij}(t)=p_{ji}(t^{-1})\]
for any $i$ and $j$. We will prove the following claim.

\begin{claim}
Let $l\in \{1,\dots,n\}$.  Then   there exist based nice  disks  $D_1,\dots,D_n$ with the following properties:
\begin{itemize}
\item[\textbf{(1)}] the disks $D_1,\dots,D_n$ are properly arranged,
\item[\textbf{(2)}] for any $i$ the based curve $c_i:=\partial D_i$ represents $x_i$,
\item[\textbf{(3)}] the disks $D_{l+1},\dots,D_n$ are disjoint,
\item[\textbf{(4)}] the first $2l$ entries of the arrangement of $D_1,\dots,D_n$ are
\[ \{1,1,2,2,\dots,l,l\},\]
\item[\textbf{(5)}] if for $i=1,\dots,n$ we equip $c_i=\partial D_i$ with the framing $p_{ii}(1)$, then
\[ \lkt(c_i,c_j)=\frac{p_{ij}(t)}{\Delta_K(t)}\in \Q(t),\]
for any $i \in\{1,\dots,l\}$ and $j\in \{1,\dots,n\}$.
\end{itemize}
\end{claim}

It is clear that the statement of the claim for $l=n$ is precisely the statement of Lemma \ref{lem:maintechnicallem}.

We will prove the claim by induction on $l$. We begin with $l=0$.
First note that Lemma~\ref{lem:fogel} allows us to find disjoint disks $D_1,\ldots,D_n$ such that (2) and (3) are satisfied. 
The remark after Definition~\ref{properly}
shows that $D_1,\ldots,D_n$ are properly arranged. Conditions (4) and (5) for $l=0$ are empty.

Now suppose that the statement of the claim holds for $l-1$.
We thus pick based nice  disks  $D_1,\dots,D_n$ which satisfy the statement of the claim for $l-1$.
We first apply  the type $R_1$ move several times to the `left most' intersection point of $D_{l}$ so  that
the first $2(l-1)+1$ entries of the arrangement of $D_1,\dots,D_n$ are
\[ \{1,1,2,2,\dots,l-1,l-1,l\}.\]
We repeat this procedure with the `right most' intersection point of $D_{l}$ so  that
after several further type $R_1$ moves
the first $2l$ entries of the arrangement of the resulting disks $D_1,\dots,D_n$ are
\[ \{1,1,2,2,\dots,l-1,l-1,l,l\}.\]
Since we applied type $R_1$ moves it follows that the disks are properly arranged.

For $i=1,\dots,n$ we equip $c_i:=\partial D_i$ with the framing $p_{ii}(1)$. We denote by  $q_{ij}(t)$, $i,j\in \{l,\dots,n\}$
the polynomials
which satisfy
\[ \lkt(c_i,c_j)=\frac{q_{ij}(t)}{\Delta_K(t)}.\]
%(Here, and throughout the proof we will always write $c_k=\partial D_k$, $k=1,\dots,n$.)
Given $s\in \{1,\dots,n\}$ we now also consider the following property:
\bn
\item[$\mathbf{(5_s)}$] for any $i\in \{1,\dots,l\}$ and $j\in \{1,\dots,s\}$ we have
\[ \lkt(c_i,c_j)=\frac{p_{ij}(t)}{\Delta_K(t)}\in \Q(t).\]
\en
%It is clear that Properties $(1)\dots (4)\,(5_{l-1})$ hold.
Note that $(5_{l-1})$ holds since the disks satisfy Property (5) for $l-1$ and since the $p_{ij}$ and $q_{ij}$ are both antisymmetric in $i$ and $j$. 
We now proceed with two steps, first we will arrange the disks  such that ($5_l$) holds,
and then we will furthermore modify the disks such that ($5_s$) holds for any $s>l$.

\textbf{(a)} Recall that by   the discussion in Section \ref{section:blanchfield} we have
\[ q_{ll}(1)=\lkt(c_l,c_l)|_{t=1}=\mbox{framing of $c_l$}=p_{ll}(1).\]
It thus follows that $q_{ll}(1)-p_{ll}(1)=0$. Note that furthermore $p_{ll}(t)=p_{ll}(t^{-1})$ by assumption and that $q_{ll}(t)=q_{ll}(t^{-1})$ by the symmetry of $l$.
It now follows that we can write
\[ q_{ll}(t)-p_{ll}(t)=\sum_{i=0}^ka_i(t^i+t^{-i})\]
for some $a_0,\dots,a_k\in \Z$ with $\sum_{i=0}^ka_i=0$.
Put differently, we can write
\[ q_{ll}(t)-p_{ll}(t)=\sum_{i=1}^k b_i(t^i-t^{i-1}-t^{-(i-1)}+t^{-i})\]
for some $b_1,\dots,b_k\in \Z$.

Considering (\ref{equ:lcii}) and (\ref{equ:kij}) it follows easily
that for $i=1,\dots,k$ we can now apply $|b_i|$ times
an appropriate combination of a type $T(n)$ move together with either a type $F_1$ move or a type $F_2$ move
to arrange that
\[ \lkt(c_l,c_l)=\frac{p_{ll}(t)}{\Delta_K(t)}\in \Q(t).\]
This concludes the proof of $(5_l)$.

\textbf{(b)} We now suppose that we have disks which satisfy Properties (1)\dots (4) and $(5_{s-1})$ for some $s-1\geq l$.
%For $i=l+1,\dots,n$ we now make modifications to the disks
%to arrange the values for $\lkt(c_l,c_i)$ and $\lkt(c_i,c_l)$, without affecting any other twisted linking numbers.
%In the following let $i>l$.
%The disk $D_l$ precedes $D_i$ and it thus follows easily that
%the linking number of $c_l$ and $c_i$ is zero.
It follows from   the discussion in Section \ref{section:blanchfield} that
\[ q_{sl}(1)=\lkt(c_s,c_l)|_{t=1}=\lk(c_s,c_l)=0=p_{sl}(1).\]
It thus follows that $q_{sl}(1)-p_{sl}(1)=0$. We can therefore write
\[ q_{sl}(t)-p_{sl}(t)=\sum_{i=-k}^k b_i(t^i-t^{i-1})\]
for some $b_{-k},\dots,b_k\in \Z$.
We now apply the type $R_2$ moves several times so that the right hand  piece of $D_l$ is adjacent to the piece of $D_s$
with the opposite orientation.
Considering (\ref{equ:lcij}) and (\ref{equ:kij}) it follows easily
that for $i=-k,\dots,k$ we can now apply $|b_i|$ times
an appropriate combination of a type $T(n)$ move together with either a type $F_1$ move or a type $F_2$ move
to arrange that
\[ \lkt(c_s,c_l)=\frac{p_{sl}(t)}{\Delta_K(t)}\in \Q(t).\]
Finally we conclude with several type $R_1$ moves so that the arrangement is unchanged.
Note that the resulting disks are again properly arranged.

After Steps (a) and (b) the resulting disks clearly have the required properties. This concludes the proof of the claim and
thus of Lemma \ref{lem:maintechnicallem}.
\end{proof}

\end{document}